\documentclass[11pt]{article}  
\usepackage{amsmath}
\usepackage{amsfonts}
\usepackage{mathtools}
\usepackage{enumerate}
\usepackage{amsthm}
\usepackage[textwidth=5.6in,textheight=8.0625in]{geometry}
\usepackage[toc,page]{appendix}

\newtheorem{theorem}{Theorem}[section]
\newtheorem{corollary}[theorem]{Corollary}
\newtheorem{lemma}[theorem]{Lemma}
\newtheorem{proposition}[theorem]{Proposition}
\newtheorem{remark}[theorem]{Remark}
\newtheorem{definition}[theorem]{Definition}
\newtheorem{example}{Example}

\newtheorem{assum}{Assumption}
\numberwithin{equation}{section}
\def\a{\alpha}				
\def\b{\beta}				
\def\eps{\varepsilon}				
\def\Sa{S_{\alpha}}			
\def\sa{s_{\alpha}}			
\def\va{v_{\alpha}}			
\def\tv{\tilde{v}}			
\def\bva{\bar{v}_{\alpha}}	
\def\bv{\bar{v}}
\def\Ga{G_{\alpha}}			
\def\bGa{\bar{G}_{\alpha}}			
\def\bG{\bar{G}}			
\def\tva{\tilde{v}_{\alpha}}

\def\tG{\tilde{G}}
\def\bGa{\bar{G}_{\alpha}}	
\def\h{h_{\alpha}}			
\def\ra{x^{\min}_{\alpha}}			
\def\rb{x^{\min}_{\beta}}

\def\rtan{x^{\min}_{\tilde{\alpha}_n}}
\def\rone{x^{\min}_1}
\def\ronea{x^{\min}_{1,*}}

\def\barS{S^*}
\def\ma{m_{\alpha}}			
\def\ua{u_{\a}}				
\def\bma{\bar{m}_\alpha}		
\def\bu{\bar{u}}
\def\bua{\bar{u}_\alpha}
\def\c{\bar{c}}				
\def\tu{\tilde{u}}			
\def\E{\mathbb{E}}			
\def\PR{P}			
\def\indF{{\rm \textbf{1}}}		
\def\R{\mathbb{R}}			
\def\Z{\mathbb{Z}}			
\def\N{\mathbb{N}}			
\def\U{\mathbb{U}}			
\def\Rp{\R_{+}}				
\def\X{\mathbb{X}}			
\def\Y{\mathbb{Y}}			
\def\A{\mathbb{A}}			
\def\PS{\Pi}					
\def\argmin{{\rm arg}\!\min}		
\def\K{\mathbb{K}}
\def\bR{\bar{\R}}		
\def\zone{z^{(1)}}				
\def\bh{\textbf{h}}
\title{\LARGE \bf
Stochastic Setup-Cost Inventory Model with Backorders and Quasiconvex Cost Functions
}
\author{
Eugene~A.~Feinberg, Yan~Liang \\
\small
\textit{Department of Applied Mathematics and Statistics} \\
\small
\textit{Stony Brook University, Stony Brook, NY 11794} \\
\small
\textit{eugene.feinberg@stonybrook.edu, yan.liang@stonybrook.edu}
}
\date{}
\begin{document}
\maketitle
%
\begin{abstract}
\textit{
This paper studies a periodic-review single-commodity setup-cost inventory model with backorders and holding/backlog costs satisfying quasiconvexity assumptions. We show that the Markov decision process for this inventory model satisfies the assumptions that lead to the validity of optimality equations for discounted and average-cost problems and to the existence of optimal $(s,S)$ policies. In particular, we prove the equicontinuity of the family of discounted value functions and the convergence of optimal discounted lower thresholds to the optimal average-cost one for some sequences of discount factors converging to $1.$ If an arbitrary nonnegative amount of inventory can be ordered,
we establish stronger  convergence properties: (i) the optimal discounted lower thresholds $s_\alpha$ converge to optimal average-cost lower threshold $s;$ and (ii) the discounted relative value functions converge to  average-cost relative value function. These convergence results previously were known only for subsequences of discount factors even for problems with convex holding/backlog costs. The results of this paper also hold for problems with fixed lead times.
}
\end{abstract}
%
%

\section{Introduction}
\label{sec:introduction}

In this paper we study a periodic-review single-commodity setup-cost inventory  model with backorders and holding/backlog costs
satisfying quasiconvexity assumptions. We show that the Markov decision process for this inventory model satisfies the assumptions that lead to the validity of optimality equations for discounted and average-cost problems and to the existence of optimal $(s,S)$ policies. In particular, we prove the equicontinuity of the family of discounted value functions and the convergence of optimal discounted lower thresholds to the optimal average-cost one for some sequences of discount factors converging to $1.$ If an arbitrary nonnegative amount of inventory can be ordered,
we establish stronger  convergence properties: (i) the optimal discounted lower thresholds $s_\alpha$ converge to optimal average-cost lower threshold $s;$ and (ii) the discounted relative value functions converge to  average-cost relative value function. These convergence
results previously were known only for subsequences of discount factors even for problems with convex holding/backlog costs.  The results of this paper hold for problems with deterministic positive lead times.


For problems with convex holding/backlog cost functions, Scarf~\cite{Sca60}
introduced the concept of $K$-convexity to prove the optimality of
$(s,S)$ policies for finite-horizon problems with continuous demand and convex holding/backlog costs.
Zabel \cite{za62} indicated some gaps in Scarf~\cite{Sca60} and corrected them.
References~\cite{BCST,BS99,CS04a,CS04b,Ftut,FL16,FLi16a,FLi16b,Igl63,Sca60,SCB05,VW65}
deal with convex or linear holding/backlog cost functions.
Iglehart~\cite{Igl63} extended Scarf's~\cite{Sca60} results to infinite-horizon
problems with continuous demand. Veinott and Wagner~\cite{VW65} proved the optimality
of $(s, S)$ policies for both finite-horizon and infinite-horizon problems with discrete
demand.  Beyer and Sethi~\cite{BS99} completed the missing proofs in Iglehart~\cite{Igl63}
and Veinott and Wagner~\cite{VW65}.
Chen and Simchi-Levi~\cite{CS04a, CS04b} studied coordinating inventory
control and pricing problems and proved the optimality of $(s,S)$ policies without assuming that
the demand is discrete or continuous.  Under certain assumptions, their results imply the optimality of $(s,S)$ policies for problems without pricing.
Beyer et al.~\cite{BCST} and Huh et al.~\cite{HJN11} studied problems with parameters depending on exogenous factors modeled by a Markov chain.  Additional references can be found in monographs by Porteus~\cite{Por02} and Zipkin~\cite{Zip00}.

The analysis of periodic-review inventory models is based on the theory of Markov Decision Processes (MDPs).
However, most of inventory control papers use only basic facts from the MDP theory, and the corresponding general results had been unavailable for a long time.
Feinberg et al.~\cite{FKZ12} developed the results on MDPs with Borel state spaces, possibly noncompact action sets, and  possibly unbounded one-step costs. Discrete-time periodic-review inventory control problems are particular examples of such MDPs; see Feinberg~\cite{Ftut} for details.  Feinberg and Lewis~\cite{FL16} obtained additional convergence results for convergence of optimal actions for MDPs and established the optimality of $(s,S)$ policies for inventory control problems as well as other results.  Feinberg and Liang~\cite{FLi16a} provided descriptions of optimal policies for all possible values of discount factors (for some parameters, optimal $(s,S)$ policies may not exist for discounted and finite-horizon problems).  Feinberg and Liang~\cite{FLi16b} proved that discrete-time periodic-review inventory models with backorders and convex holding/backlog costs satisfy the equicontinuity assumption, and this implies several additional properties of optimal average-cost policies including the validity of average-cost optimality equations (ACOEs).

Veinott~\cite{Vei66}
 studied the nonstationary setup-cost inventory model
with a fixed lead time, backorders, and  holding/backlog costs satisfying quasiconvexity assumptions.
Veinott~\cite{Vei66} proved the optimality of $(s,S)$ policies for finite-horizon problems
and also provided bounds on the values of the optimal thresholds $s$ and $S.$
Zheng~\cite{Zheng} proved the optimality of $(s,S)$ policies for
models with quasiconvex cost functions and discrete demand under both discounted and
average cost criteria by constructing a solution to the optimality equations.

In this paper we consider the infinite-horizon stationary inventory model with holding/backlog costs satisfying
quasiconvexity assumptions. These quasiconvexity assumptions are introduced by Veinott~\cite{Vei66} for
finite-horizon nonstationary models. Zheng~\cite{Zheng} and Chen and Simchi-Levi~\cite{CS04b} considered a slightly
stronger quasiconvexity assumption for infinite-horizon stationary models. For inventory  model with holding/backlog
costs satisfying quasiconvexity assumptions, this paper establishes convergence properties of optimal discounted
thresholds for discounted problems to the corresponding thresholds for average-cost problems. Some of the results are
new even for problems with convex holding/backlog costs. While convergence of optimal thresholds and relative
discounted value functions was known only for subsequences of discount factors (see \cite{BCST,FL16, FLi16b,HJN11}),
here we show that convergence of lower thresholds and discounted value functions takes place for all
discount factors tending to 1.

The rest of the paper is organized in the following way.
Section~\ref{sec:inventory control}  describes the setup-cost
inventory  model and introduces the assumptions used in this paper.
Section~\ref{sec:mdptotal} establishes the optimality of $(\sa,\Sa)$ policies for
the infinite-horizon problem with the discount factor $\a.$
Section~\ref{sec:mdpac} verifies average-cost optimality assumptions and
the equicontinuity conditions for discounted relative value functions.
Section~\ref{sec:acoe} establishes the validity of ACOEs for the inventory  model
and the optimality of $(s,S)$ policies under the average cost criterion.
Section~\ref{sec:cvg sa} establishes the convergence of discounted optimal
lower thresholds $\sa,$ when the discount factor $\a$ converges to $1,$ to the average-cost optimal lower threshold $s.$
Section~\ref{sec:cvg ua} establishes the convergence of discounted relative
value functions, when the discount factor converges to $1.$
Section~\ref{sec:veinott} presents a reduction from the inventory  model with constant
lead times to the model without lead times using Veinott's~\cite{Vei66} approach.

\section{Setup-Cost Inventory Model with Backorders: Definitions and Assumptions}
\label{sec:inventory control}

Let $\R$ denote the real line, $\Z$ denote the set of all integers, $\Rp:=[0,+\infty)$ and
$\N_0 := \{0,1,2,\ldots \}.$ Consider the stochastic
periodic-review setup-cost inventory  model with backorders.
At times $t=0,1,\ldots,$ a decision-maker views the current inventory of a single commodity
and makes an ordering decision. Assuming zero lead times, the products are immediately available
to meet demand. The cost of ordering is incurred at the time of delivery of the order.
Demand is then realized, the decision-maker views the remaining
inventory, and the process continues. The unmet demand is backlogged. The
demand and the order quantity are assumed to be nonnegative. The objective is to minimize the infinite-horizon
expected total discounted cost for discount factor $\a\in (0,1)$ and long run average cost for $\a = 1.$
The inventory  model is defined by the following parameters:
\begin{enumerate}
    \item $K > 0$ is a fixed ordering cost;
    \item $\bar{c}>0$ is the per unit ordering cost;
    \item $\{D_t,t=1,2,\dots\}$ is a sequence of i.i.d.\ nonnegative finite
    random variables representing the demand at periods $0,1,\dots\ .$
    We assume that $\E [D] < +\infty$ and $\PR(D>0)>0,$
    where $D$ is a random variable with the same distribution as $D_1;$
    \item $h(x)$ is the holding/backlog cost per period if the inventory level is $x.$
    Assume that: (i) the function $\E[h(x-D)]$
    is finite and continuous for all $x\in\X;$
    and (ii) $\E[h(x-D)]\to +\infty$ as $|x|\to +\infty.$
\end{enumerate}
Without loss of generality,  assume that the function $\E[h(x-D)]$ is nonnegative.
The assumption $\PR(D>0)>0$ avoids the trivial case when there is no demand.

Now we formulate an MDP for this inventory  model.
The state and action spaces can be either (i) $\X=\R$ and $\A=\Rp;$ or (ii)
$\X=\Z$ and $\A = \N_0,$ if the demand $D$ takes only integer values and only integer orders are allowed.

The dynamics of the system are defined by the equation
\begin{align}
	x_{t+1} = x_t + a_{t} - D_{t+1}, \qquad t=0,1,2,\ldots,
	\label{eqn:dynamic2}
\end{align}
where $x_t$ and $a_t$ denote the current inventory level and the
ordered amount at period $t$ respectively.
The transition probability $q(d x_{t+1} | x_t, a_t)$ for the MDP defined by the
stochastic equation \eqref{eqn:dynamic2} is
\begin{align}
	q(B | x_t, a_t) = \PR (x_t + a_t - D_{t+1} \in B)
	\label{eqn:tranp2}
\end{align}
for each measurable subset $B$ of $\R.$ The one-step expected cost is
\begin{align}
	c(x,a) := K \indF_{\{ a > 0 \}} + & {\bar c}a + \E[h(x+a-D)], \qquad (x,a)\in \X\times\A,
	\label{eqn:c}
\end{align}
where $\indF_B $ is an indicator of the event $B.$

Let $H_t = (\X\times\A)^{t}\times\X$
be the set of histories for $t=0,1,\dots\ .$ Let $\PS$ be the set of all policies. A
(randomized) decision rule at period $t=0,1,\dots$ is a regular transition probability
$\pi_t : H_t\to \A,$ 
that is, (i) $\pi_t(\cdot|\bh_t)$ is a probability distribution on $\A,$
where $\bh_t=(x_0,a_0,x_1,\dots,a_{t-1},x_t),$ and (ii) for any measurable subset
$B\subset \A,$ the function $\pi_t(B|\cdot)$ is measurable on $H_t.$
A policy $\pi$ is a sequence $(\pi_0,\pi_1,\dots)$ of decision rules.
Moreover, $\pi$ is called non-randomized if each probability measure $\pi_t(\cdot|\bh_t)$ is
concentrated at one point. A non-randomized policy is called stationary if all decisions depend
only on the current state.
According to the Ionescu Tulcea theorem
(see  Hern\'{a}ndez-Lerma and Lasserre~\cite[p.~178]{HLL96}),
given the initial state $x,$ a policy $\pi$ defines the probability distribution
$\PR_{x}^{\pi}$ on the set of all trajectories $H_{ +\infty} = (\X\times\A)^{ +\infty}$.
We denote by $\E_{x}^{\pi}$ the expectation with respect to $\PR_{x}^{\pi}.$

For a finite-horizon $N=0,1,\dots,$ let us define the expected total discounted costs
\begin{equation}\label{eqn:sec_model def:finite total disc cost}
    v_{N,\a}^{\pi} (x) := \mathbb{E}_{x}^{\pi} \Big[ \sum_{t=0}^{N-1}
    \alpha^{t} c(x_t,a_t)  \Big] , \qquad x\in\X,
\end{equation}
where $\alpha\in [0,1]$ is the discount factor and $v_{0,\a}^{\pi} (x)= 0,$ $x\in\X.$
When $N= +\infty$ and $\alpha\in [0,1),$
\eqref{eqn:sec_model def:finite total disc cost} defines the
infinite-horizon expected total discounted cost denoted by $v_{\a}^{\pi}(x).$ Let $\va (x):=\inf_{\pi\in\PS} \va^\pi(x),$ $x\in\X.$  A policy $\pi$ is called optimal for the respective criterion with discount factor $\a$
if $v_{N,\a}^{\pi}(x)=v_{N,\a}(x)$ or $v^\pi_\a(x)=\va (x)$ for all $x\in\X.$

The \emph{average cost per unit time} is defined as
\begin{equation}\label{eqn:sec_model def:avg cost}
    w^{\pi}(x):=\limsup_{N\to +\infty} \frac{1}{N}v_{N,1}^{\pi} (x), \qquad x\in\X.
\end{equation}

Define the optimal value function $w^{{\rm ac}}(x):=\inf_{\pi\in\PS} w^{\pi} (x),$  $x\in\X.$
 A policy $\pi$ is called
average-cost optimal if $w^{\pi}(x)=w^{{\rm ac}}(x)$ for all $x\in\X.$

Recall the definition of quasiconvex functions.
\begin{definition}\label{def:qc}
	A function $f$ is quasiconvex on a convex set $X\subset \R,$ if for all $x,$ $y \in  X$ and $0 \leq \lambda \leq 1$
	\begin{align*}
	 f(\lambda x + (1-\lambda) y ) \leq \max\{ f(x), f(y) \}.
	\end{align*}
\end{definition}

For $\a\in (0,1],$ let us define
\begin{align}
	\h (x)  := h (x) + (1-\a) \c x + \c \E[D], \qquad x\in\X.
	\label{eqn:def ha}
\end{align}
Note that since $\E[h(x-D)]\to +\infty$ as $x\to  +\infty$ and $(1-\a)\c\geq 0$
for all $\a\in(0,1],$ the function $\E[\h (x-D)] = \E[h(x-D)] + (1-\a)\c x + \a \c \E[D]$ tends to $+\infty$ as
$x\to +\infty$ for all $\a\in[0,1].$ In addition, for $\a\in (0,1]$ the function $\E[\h (x-D)]$ is continuous on $\X$ because the
functions $\E[h (x-D)]$ and $(1-\a)\c x$ are continuous on $\X.$

Consider the following assumptions on the quasiconvexity or convexity of the cost function.
\begin{assum}\label{assum:qcall}
There exists $\a^*\in[0,1)$ such that for all $\a\in (\a^*,1]:$
\begin{enumerate}[(i)]
	\item \label{assum:qcall1} The function $\E[\h (x-D)]$ is quasiconvex; and 
	\item \label{assum:qcall2} $\lim_{x\to -\infty} \E[\h (x-D)] > K + \inf_{x\in\X} \{ \E[\h (x-D)] \}.$ 
\end{enumerate}
\end{assum}

\begin{assum}\label{assum:convex}
	The function $h(\cdot)$ is convex on $\X.$
\end{assum}

For the discounted criterion, consider the following assumption, which is weaker than
Assumption~\ref{assum:qcall}. Assumption~\ref{assum:qcall} is used for the convergence of discounted-cost problems to average-cost problem.

\begin{assum}\label{assum:qcalpha}
For a given $\a\in(0,1]$ assume that:
\begin{enumerate}[(i)]
	\item \label{assum:qcalpha:1} the function $\E[\h (x-D)]$ is quasiconvex; and
	\item \label{assum:qcalpha:2} $\lim_{x\to -\infty} \E[\h (x-D)] > K + \inf_{x\in\X} \{ \E[\h (x-D)] \}.$
\end{enumerate}
\end{assum}

\noindent
We recall that Veinott~\cite{Vei66} considered quasiconvexity assumptions for finite-horizon nonstationary problems.
Being applied to stationary infinite-horizon problems, the corresponding assumption is Assumption~\ref{assum:qcalpha}.
For stationary infinite-horizon models and discrete demands, Zheng~\cite{Zheng} used a slightly stronger assumption,
which is Assumption~\ref{assum:qcalpha} with \eqref{assum:qcalpha:2} replaced with
$\lim_{x\to -\infty} \E[h_{\a} (x-D)] = +\infty.$

For $\a\in[0,1],$ if
\begin{align}
	\lim_{x\to -\infty} \E[h_\a (x-D)] > \inf_{x\in\X} \E[h_\a (x-D)] ,
	\label{eqn:raexist}
\end{align}
then we define
\begin{align}
	\ra := \min \big\{ \underset{x\in\X}\argmin \{ \E[\h (x-D)] \} \big\} .
	\label{eqn:def r alpha r*}
\end{align}
Since the function $\E[\h (x-D)]$ is continuous, $\E[\h (x-D)]\to +\infty$ as $x\to +\infty$ and
\eqref{eqn:raexist} imply that $|\ra| <  +\infty.$

The following assumption is used to establish the convergence of the discounted optimal lower thresholds
and relative value functions in Sections~\ref{sec:cvg sa} and \ref{sec:cvg ua} respectively.

\begin{assum}\label{assum:decrease}
For a given $\a\in (0,1],$ the function $\E[h_{\a} (x-D)]$ is strictly decreasing on $(-\infty,\ra],$ where $\ra$ is defined in \eqref{eqn:def r alpha r*}.
\end{assum}

We state the relationships between these assumptions in the following two lemmas. The proofs of the lemmas presented in
this section are available in Appendix~\ref{sec:inventory control:apx}.

\begin{lemma}\label{lm:cqc}
Assumption~\ref{assum:convex} implies the validity of Assumption~\ref{assum:qcall} with
\begin{align}
	\a^* \in [ \max\{ 1 + \lim_{x\to -\infty} \frac{h(x)}{\c x} ,0\} , 1)
	\label{eqn:convexa*}
\end{align}
and the validity of Assumption~\ref{assum:decrease} for all $\a\in(\a^*,1].$
\end{lemma}

\begin{lemma}\label{lm:assumption}
Assumption~\ref{assum:qcall} implies Assumption~\ref{assum:qcalpha} for $\a\in (\a^*,1].$
\end{lemma}

\section{Setup-Cost Inventory  Model with Discounted Costs}
\label{sec:mdptotal}

This section establishes the existence of optimal $(\sa,\Sa)$ polclies for the problems
with discounted costs stated in Theorem~\ref{thm:qcall}. We start this section by
verifying the weak continuity of the transition probability $q$ defined in \eqref{eqn:tranp2}
and the $\K$-inf-compactness of the one-step cost function $c$ defined in \eqref{eqn:c}. These
properties, stated in Assumption~\textbf{W*}, imply the validity of optimality equations and
the convergence of value iterations for problems with discounted costs; see
Feinberg~et~al.~\cite[Theorem~4]{FKZ12}.

Recall that a function $f:U\to \R\cup \{+\infty\},$ where $U$ is a subset of a metric space $\U,$ is
called inf-compact, if for every $\lambda\in\R$ the level set $\{u\in\U :f(u)\leq \lambda \}$
is compact.
\begin{definition}[{Feinberg et al.~\cite[Definition 1.1]{FKZ13}, Feinberg~\cite[Definition 2.1]{Ftut}}]
\label{def:k inf compact}
	A function $f:\X\times\A\to \bR$ is called $\K$-inf-compact, if for
	every nonempty compact subset $K$ of $\X,$ the function $f:K\times\A\to \bR$ is inf-compact.
\end{definition}

It is known for discounted MDPs that if the one-step cost
function $c$ and transition probability $q$ satisfy the Assumption~\textbf{W*} below, then it is possible
to write the optimality equations for the finite-horizon and infinite-horizon problems,
these equations define the sets of stationary and Markov optimal policies for
infinite and finite horizons respectively,
$v_{\a} (x) = \lim_{N\to +\infty} v_{N,\a} (x)$ for all $x\in\X,$ and the functions $v_{N,\a},$
$N = 1,2,\ldots,$ and $v_{\a}$ are lower semicontinuous;
see Feinberg~et~al.~\cite[Theorems~3,~4]{FKZ12}.

\vspace{0.1in}
\noindent
\textbf{Assumption W*} ({\rm Feinberg~et~al.~\cite{FKZ12}, Feinberg and Lewis~\cite{FL16},
or Feinberg~\cite{Ftut}})\textbf{.}

(i) The function $c$ is $\K$-inf-compact and bounded below, and

(ii) the transition probability $q(\cdot|x,a)$ is weakly continuous in
  $(x,a)\in \X\times\A,$  that is, for every bounded continuous function $f:\X\to\R,$ the function
  $\tilde{f}(x,a):=\int_\X f(y)q(dy|x,a)$ is continuous on $\X\times\A.$
\vspace{0.1in}

\begin{theorem}\label{thm:W*}
The inventory  model satisfies Assumption~\textbf{W*},
and the one-step cost function $c$ is inf-compact.
\end{theorem}

\begin{proof}
Since the function $\E[h(x-D)]$ is continuous and tends to $ +\infty$ as $|x|\to +\infty,$
Theorem~\ref{thm:W*} is proved in Feinberg and Lewis~\cite[Corollary 5.2]{FL16} and Feinberg~\cite[p.~22]{Ftut}
\end{proof}

According to Feinberg and Lewis~\cite{FL16}, since Assumption~\textbf{W*} holds for
the MDP corresponding to the described inventory  model, the optimality equations for
the total discounted costs can be written as
\begin{align}
	v_{t+1,\a} (x) &= \min \{ \min_{a\geq 0}[ K + G_{t,\alpha}(x+a)], G_{t,\alpha}(x) \} - \bar{c}x, \quad t=0,1,2,\ldots,\  x\in\X, \label{eqn:vna} \\
	\va (x) & = \min \{ \min_{a\geq 0}[ K + G_{\alpha}(x+a)], G_{\alpha}(x) \} - \bar{c}x,\qquad x\in\X, \label{eqn:va}
\end{align}
where
\begin{align}
G_{t, \alpha}(x)  &:= {\bar c} x + \E[h(x-D)] + \alpha \E [v_{t, \alpha}(x-D)], \quad t=0,1,2,\ldots,\  x\in\X, \label{eqn:gna} \\
G_{\alpha}(x)  &:= {\bar c} x + \E[h(x-D)] + \alpha \E [v_{\alpha}(x-D)],\qquad  x\in\X, \label{eqn:ga}
\end{align}
and $v_{0,\a}(x) = 0$ for all $x\in\X.$

Recall the definition of $(s,S)$ policies.
Suppose $f(x)$ is a lower semicontinuous function such that
$\liminf_{|x|\to +\infty}f(x)> K + \inf_{x\in\X} f(x).$ Let
\begin{align}
	S &\in \underset{x\in\X}{\argmin} \{ f (x) \} \label{eqn:def S}, \\
	s &= \inf \{ x\leq S: f (x) \leq K + f(S) \}. \label{eqn:def s}
\end{align}

\begin{definition}\label{def:sS policy}
	Let $s_t$ and $S_t$ be real numbers such that $s_t\leq S_t,$ $t=0,1,\ldots\ .$
	A policy is called
	an $(s_t,S_t)$ policy at step $t$ if it orders up to the level $S_t,$ if $x_t<s_t,$
	and does not order, if $x_t\geq s_t.$ A Markov policy is called an $(s_t,S_t)$ policy
	if it is an $(s_t,S_t)$ policy at all steps $t=0,1,\ldots\ .$ A policy is called an
	$(s,S)$ policy if it is stationary and it is an $(s,S)$ policy at all steps
	$t=0,1,\ldots\ .$
\end{definition}

In this section, we consider Assumption~\ref{assum:qcalpha}, which guarantees the optimality
of $(\sa,\Sa)$ policies for infinite-horizon problems with the discount factor $\a,$
as this is stated in the following theorem.

\begin{theorem}\label{thm:qcall}
Let Assumption~\ref{assum:qcalpha} hold for $\a\in(0,1).$ For the infinite-horizon problem,
there exists an optimal $(s_\alpha, S_\alpha)$ policy, where $S_{\a}$ and $s_{\a}$ are real numbers
such that $S_{\a}$ satisfies  \eqref{eqn:def S}  and 	$s_{\a}$ is defined in \eqref{eqn:def s}
with $f(x):=G_{\a}(x),$ $x\in\X.$
\end{theorem}

\begin{remark}\label{rm:sSconvex}
{\rm
Under slightly stronger assumption, this theorem is prove by Zheng~\cite{Zheng}
for inventory models with integer demands and integer orders.
Under Assumption~\ref{assum:convex} and some other technical assumptions,
this conclusion also follows from Chen and Simchi-Levi~\cite{CS04b}.
Under Assumption~\ref{assum:convex}, Theorem~\ref{thm:qcall} is proved in
Feinberg and Liang~\cite[Theorem 4.4]{FLi16a} with $\a\in(\a^*,1)$ for $\a^*$ defined in \eqref{eqn:convexa*}.
In addition, a simple structure of optimal polices is described in Feinberg and Liang~\cite[Theorem 4.4]{FLi16a}
for all $\a\in[0,1)$. However, under Assumption~\ref{assum:qcalpha}
in this paper, if $\a\in [0,\a^*),$ then the structure of optimal policies is currently not clear.
}
\end{remark}

To prove the optimality of $(\sa,\Sa)$ policies, we first consider the same inventory model with
a terminal cost $ -\c x,$ that is, each unit of stock left over can be discarded with the return of $\c$ and each
unit of backlogged demand is satisfied at the cost $\c.$ For this model with terminal costs, the one step cost function
are the same as the original problem and let us denote the expected total discounted cost by
\begin{equation}\label{eqn:sec_model def:finite total disc cost:t}
    \tv_{N,\a}^{\pi} (x) := \mathbb{E}_{x}^{\pi} \Big[ \sum_{t=0}^{N-1}
    \alpha^{t} c(x_t,a_t) -\a^N \c x_N  \Big] , \qquad x\in\X .
\end{equation}
Then we transform the problem into the one with $\c = 0$ and follow the induction proofs in Veinott~\cite{Vei66} to establish properties for $\tv_{\a}$. Then, we shall also show that $\tv_{\a} = \va.$

The finite-horizon discounted cost optimality equations for the inventory model with
terminal costs $-\c x$ can be written as
\begin{align}
	\tv_{t+1,\a} (x) &= \min \{ \min_{a\geq 0}[ K + \tG_{t,\alpha}(x+a)], \tG_{t,\alpha}(x) \} - \bar{c}x, \quad t=0,1,2,\ldots,\  x\in\X, \label{eqn:vna:t}
\end{align}
where
\begin{align}
\tG_{t, \alpha}(x)  &:= {\bar c} x+\E[h(x-D)]+\alpha\E [\tv_{t,\alpha}(x-D)],\quad t=0,1,2,\ldots,\  x\in\X \label{eqn:gna:t}
\end{align}
and $\tv_{0,\a}(x) = -\c x$ for all $x\in\X.$
Then, following Veinott~\cite{Vei66}, we transform the model with positive unit ordering cost and terminal costs $-\c x$ into the model with zero unit and terminal costs and holding/backlog costs $\h$ defined in \eqref{eqn:def ha}.
The one-step cost function for the new model is
\begin{align}\label{eqn:c:trans}
	c_\a (x,a) = K\indF_{\{a>0 \}} +\E[\h (x+a-D)].
\end{align}
Since the function $\E[\h (x-D)]$ is quasiconvex, $\lim_{x\to-\infty}\E[\h (x-D)] >  K + \inf_{x\in\X} \E[\h (x-D)],$
and $\lim_{x\to+\infty} \E[\h (x-D)]=+\infty,$ the function $\E[\h (x-D)]$ is bounded below.
Therefore, $c_\a$ is bounded below and the new model satisfies Assumption~\textbf{W*}.
The optimality equations for the new model are
\begin{align}
	\bv_{t+1,\a} (x) & = \min \{ \min_{a\geq 0} [K + \bG_{t,\a} (x+a)], \bG_{t,\a} (x) \}, \quad t = 0,1,2,\ldots, \  x\in\X, \label{eqn:vna:trans} \\
	\bva (x) &= \min \{ \min_{a\geq 0} [K + \bGa (x+a)], \bGa(x) \}, \qquad x\in\X, \label{eqn:va:trans}
\end{align}
where
\begin{align}
	\bG_{t,\a} (x)  & = \E[\h (x-D)] + \a \E[\bv_{t,\a} (x-D)], \quad t = 0,1,2,\ldots,\  x\in\X, \label{eqn:gna:trans} \\
	\bGa (x) & = \E[\h (x-D)] + \a \E[\bva (x-D)] , \qquad x\in\X, \label{eqn:ga:trans}
\end{align}
and $\bv_{0,\a} (x) = \tv_{0,\a} (x)+\c x = 0$ for all $x\in\X.$
It is easy to see by induction that
\begin{align}
	\bv_{t,\a} (x) & = \tv_{t,\a} (x)+\c x, \qquad x\in\X, \label{eqn:trans:f} \\
	\bG_{t,\a} (x) & = \tG_{t,\a} (x), \qquad t = 0,1,2,\ldots,\  x\in\X.
	\label{eqn:bga=tga}
\end{align}
Since the validity of Assumption~\textbf{W*} for the model with zero unit cost implies that
$\bv_{t,\a} \to \bva$ as $t\to +\infty,$ in view of \eqref{eqn:trans:f} and \eqref{eqn:bga=tga}, we can define
\begin{align}
	\tva (x) := \lim_{t\to+\infty} \tv_{t,\a} (x) = \lim_{t\to+\infty} \bv_{t,\a} (x) - \c x = \bva (x) - \c x . \qquad x\in\X . \label{eqn:trans}
\end{align}
In view of \eqref{eqn:bga=tga}, the finite-horizon model with terminal costs $-\c x$ and the finite-horizon model with zero unit and terminal costs have the same sets of optimal actions for the same state-time pairs. In addition, \eqref{eqn:va:trans} implies that
\begin{align}
	\tva (x) & = \min \{ \min_{a\geq 0}[ K + \tG_{\alpha}(x+a)], \tG_{\alpha}(x) \} - \bar{c}x , \label{eqn:va:t}
\end{align}
where, in view of \eqref{eqn:ga:trans},
\begin{align}
\tG_{\alpha}(x)= \bG_{\alpha} (x)  = {\bar c} x + \E[h(x-D)] + \alpha \E [\tv_{\alpha}(x-D)] , \qquad x\in\X. \label{eqn:ga:t}
\end{align}

Now, we extend the properties of finite-horizon value functions $\bv_{t,\a}$ and $\bG_{t,\a},$ $t = 0,1,2,\ldots,$ stated
in Veinott~\cite[Lemmas~1 and 2]{Vei66} to infinite-horizon value functions $\bva$ and $\bGa.$ The proofs of the
lemmas presented in this section are available in Appendix~\ref{sec:mdptotal:apx}.

\begin{lemma}\label{lm:sS1}
For $x\leq y$ and $t=1,2,\ldots$
\begin{align}
	\bv_{t,\a} (x)  & \leq \bv_{t,\a} (y)   + K, \label{eqn:vta1} \\
	\bva (x)  & \leq \bva (y)  + K, \label{eqn:va1} \\
	\bG_{t,\a} (y) - \bG_{t,\a} (x) & \geq \E[\h(y-D)] - \E[\h(x-D)] - \a K, \label{eqn:Gta1} \\
	\bG_{\a} (y) - \bG_{\a} (x) & \geq \E[\h(y-D)] - \E[\h(x-D)] - \a K. \label{eqn:Ga1}
\end{align}
\end{lemma}

\begin{lemma}\label{lm:order}
Let Assumption~\ref{assum:qcalpha} hold for $\a\in(0,1).$ Then for $t =0,1,\ldots$ and $x\leq y \leq \ra,$
where $\ra$ is defined in \eqref{eqn:def r alpha r*},
\begin{align}
	\bG_{t,\a} (y) - \bG_{t,\a} (x) & \leq 0, \label{eqn:Gtade1} \\
	\bv_{t,\a} (y)  - \bv_{t,\a} (x) &  \leq 0. \label{eqn:vtade1}
\end{align}
\end{lemma}

To prove the optimality of $(\sa,\Sa)$ policies for infinite-horizon problems, we establish the same properties of
the infinite-horizon value functions $\bv_{\a}$ and $\bG_{\a}$ in the following lemma.

\begin{lemma}\label{lm:vaorder}
Let Assumption~\ref{assum:qcalpha} hold for $\a\in(0,1).$ Then for $x\leq y \leq \ra$
\begin{align}
	\bv_{\a} (y)  - \bv_{\a} (x) & \leq 0, \label{eqn:vade1} \\
	\bG_{\a} (y) - \bG_{\a} (x) & \leq  0. \label{eqn:Gade1}
\end{align}
\end{lemma}

\begin{theorem}\label{thm:qcall:t}
Let Assumption~\ref{assum:qcalpha} hold for $\a\in(0,1).$ For the inventory model with zero unit and terminal costs,
the following statements hold:
\begin{enumerate}[(i)]
	\item \label{thm:qcall:t:1} For $N=1,2,\ldots$ horizon problem, there exists an optimal $(s_{t,\a},S_{t,\a})_{t=0,1,2,\ldots,N-1}$ policy, where $S_{t,\a}$ and $s_{t,\a}$ are real numbers
such that $S_{t,\a}$ satisfies  \eqref{eqn:def S}  and 	$s_{t,\a}$ is defined in \eqref{eqn:def s}
with $f(x):=\bG_{N-t-1,\a}(x),$ $t=0,1,2,\ldots,N-1,$ $x\in\X;$
	\item \label{thm:qcall:t:2} For infinite-horizon problem, there exists an optimal $(\sa, \Sa)$ policy, where $S_{\a}$ and $s_{\a}$ are real numbers
such that $S_{\a}$ satisfies  \eqref{eqn:def S}  and 	$s_{\a}$ is defined in \eqref{eqn:def s}
with $f(x):=\bG_{\a}(x),$ $x\in\X;$
	\item \label{thm:qcall:t:3} For all $s_{t,\a},$ $S_{t,\a},$ $t=0,1,2,\ldots,$ and $\Sa$ defined in (i) and (ii),
	\begin{align}
		s_{t,\a} \leq \ra \leq S_{t,\a} \leq \barS_\a \quad \text{and} \quad \sa \leq \ra \leq \Sa \leq \barS_\a,
		\label{eqn:ubdStaSa}
	\end{align}	
	where
	\begin{align}
		\barS_\a := \inf\{x>\ra :   \E[\h (x-D)]\geq K + \E[\h (\ra -D)] .
		\label{eqn:defbarS}
	\end{align}
\end{enumerate}


\end{theorem}

\begin{proof}[Proof of Theorem~\ref{thm:qcall:t}]
\eqref{thm:qcall:t:1} Consider $N = 1,2,\ldots$ and $t=0,1,2,\ldots,N-1.$ Since Assumption~\textbf{W*} holds in view of Theorem~\ref{thm:W*}, Feinberg et al.~\cite[Theorem 2]{FKZ12}
implies that $\bv_{N-t-1,\a} (x)$ and $\bG_{N-t-1,\a} (x)$ are lower semicontinuous functions. In view of Lemma~\ref{lm:order},
the function $\bG_{N-t-1,\a} (x)$ is nonincreasing on $(-\infty,\ra],$ where $\ra$ is defined in \eqref{eqn:def r alpha r*}. In view of \eqref{eqn:gna:trans}, $\bG_{N-t-1,\a} (x) \geq \c x\to  +\infty$ as $x\to +\infty.$
Therefore, $\bG_{N-t-1,\a} (x)$ is inf-compact; see the definition of inf-compact functions in the paragraph before Definition~\ref{def:k inf compact}.
In view of \eqref{eqn:gna:trans} and \eqref{eqn:vtade1}, $\bG_{N-t-1,\a} (x) \geq \E[\h (x-D)] + \a \E[\bv_{N-t-1,\a} (\ra - D)]$ for all $x\leq \ra.$ Therefore,
\begin{align}
\begin{split}
	\liminf_{x\to -\infty} \bG_{N-t-1,\a} (x) & \geq \lim_{x\to -\infty} \E[\h (x-D)] + \a \E[\bv_{N-t-1,\a} (\ra - D)] \\
	&> K + \E[\h (\ra-D)] + \a \E[\bv_{N-t-1,\a} (\ra - D)] \\
	&= K + \bG_{N-t-1,\a} (\ra) \geq K + \inf_{x\in\X} \bG_{N-t-1,\a} (x) ,
\end{split}
	\label{eqn:limGa>KinfGa}
\end{align}
where the first inequality follows from the equation in the previous sentence,
the second inequality follows from Assumption~\ref{assum:qcalpha}, the equality follows from the definition of the function $\bG_{N-t-1,\a}$ in \eqref{eqn:ga:trans}, and the last inequality is straightforward.

Let $S_{t,\a}$ satisfy \eqref{eqn:def S} and $s_{t,\a}$ be defined in \eqref{eqn:def s} with $f:=\bG_{N-t-1,\a}.$
The lower semicontinuity of $\bG_{N-t-1,\a} (x)$ implies that
\begin{align}
	\bG_{N-t-1,\a}(s_{t,\a}) \leq \bG_{N-t-1,\a} (S_{t,\a}) + K.
	\label{eqn:Ga(sa)<GK}
\end{align}
Since $\bG_{N-t-1,\a} (x)$ is nonincreasing on $(-\infty,\ra],$
\begin{align}
	S_{t,\a} \geq \ra.
	\label{eqn:Sara}
\end{align}

To prove the optimality of $(s_{t,\a},S_{t,\a})$ policies, we consider three cases: (1) $x\geq \ra;$ (2) $s_{t,\a}\leq x\leq \ra;$ and (3) $x<s_{t,\a}.$

(1) In view of Lemma~\ref{lm:sS1}, for $\ra\leq x<y$
\begin{align}
	\bG_{N-t-1,\a} (y) + K - \bG_{N-t-1,\a} (x) \geq \E[\h(y-D)] - \E[\h(x-D)] + K - \a K > 0,
	\label{eqn:gantorder1}
\end{align}
where the first inequality follows from \eqref{eqn:Gta1} and the second one holds because the function
$\E[\h(x-D)]$ is nondecreasing on $[\ra, +\infty)$ and $K-\a K >0.$ Therefore, the action $a=0$ is optimal for
$x\geq \ra.$ In addition, \eqref{eqn:gantorder1} implies that $\bG_{N-t-1,\a} (x)<\bG_{N-t-1,\a} (S_{t,\a}) + K$
for all $x\in[\ra,S_{t,\a}],$ which implies that
\begin{align}
s_{t,\a} \leq \ra.
\label{eqn:sara}
\end{align}

(2) For $s_{t,\a} \leq x \leq \ra$
\begin{align*}
 \bG_{N-t-1,\a} (x) \leq \bG_{N-t-1,\a} (s_{t,\a}) \leq K + \bG_{N-t-1,\a} (S_{t,\a}) = K + \min_{y\in\X} \bG_{N-t-1,\a} (y),
\end{align*}
where the first inequality follows from \eqref{eqn:Gtade1} and the second one follows from \eqref{eqn:Ga(sa)<GK}. Therefore, the action $a=0$ is optimal for
$s_{t,\a} \leq x\leq \ra.$

(3) For $x<s_{t,\a}$
\begin{align*}
 \bG_{N-t-1,\a} (x) > K + \bG_{N-t-1,\a} (S_{t,\a}) = K + \min_{y\in\X} \bG_{N-t-1,\a} (y),
\end{align*}
where the inequality follows from the definition of $s_{t,\a}$ in \eqref{eqn:def s} with $f:=\bG_{N-t-1,\a}.$
Therefore, the action $a=S_{t,\a}-x$ is optimal for $x<s_{t,\a}.$
Thus, for $N$-horizon problem the $(s_{t,\a},S_{t,\a})_{t=0,1,2,\ldots,N-1}$ policy is optimal.

\eqref{thm:qcall:t:2} In view of Lemmas~\ref{lm:sS1}, \ref{lm:order}, and \ref{lm:vaorder}, $\bGa$ and $\bva$ satisfy
the same properties as $\bG_{t,\a}$ and $\bv_{t,\a}.$ Therefore, statement \eqref{thm:qcall:t:2} follows from
the same arguments in the proof of \eqref{thm:qcall:t:1} with $\bG_{N-t-1,\a},$ $\bv_{N-t-1,\a},$ $s_{t,\a},$ and
$S_{t,\a}$ replaced with $\bGa,$ $\bva,$ $\sa,$ and $\Sa$ respectively.

\eqref{thm:qcall:t:3} In view of \eqref{eqn:defbarS}, $\lim_{x\to+\infty} \E[\h (x-D)] = +\infty$ and
Assumption~\ref{assum:qcalpha} imply that $|\barS_\a| < +\infty$ and for $x > \barS_\a$
\begin{align}
	\E[\h(x-D)] \geq \E[\h (\barS_\a - D)] \geq K + \E[\h (\ra - D)].
	\label{eqn:x>barS:1}
\end{align}
Therefore, \eqref{eqn:Gta1} and \eqref{eqn:x>barS:1} imply that for $t=0,1,2,\ldots$ and $x>\barS_\a$
\begin{align}
\begin{split}
	\bG_{t,\a} (x) - \bG_{t,\a} (\ra) &\geq \E[\h(x-D)] - \E[\h(\ra-D)] - \a K \\
	&\geq K - \a K > 0,
\end{split}
	\label{eqn:x>barS:2}
\end{align}
which is equivalent to for $t=0,1,2,\ldots$ and $x>\barS_\a$
\begin{align}
	\bG_{t,\a} (x) > \bG_{t,\a} (\ra) \geq \min_{x\in\X} \bG_{t,\a} (x).
	\label{eqn:x>barS:3}
\end{align}
Therefore, if $x\in \argmin\{\bG_{t,\a} (x)\},$ $t=0,1,2,\ldots,$ then $x \leq \barS_\a.$ Thus, $S_{t,\a} \leq \barS_\a,$
$t=0,1,2,\ldots\ .$ In addition, the same arguments with \eqref{eqn:Gta1} and $\bG_{t,\a}$ replaced with
\eqref{eqn:Ga1} and $\bGa$ imply that $\Sa \leq \barS_\a.$

Furthermore, \eqref{eqn:Sara} and \eqref{eqn:sara} imply that
$s_{t,\a} \leq \ra \leq S_{t,\a}$ and the same arguments before \eqref{eqn:Sara} and \eqref{eqn:sara} being applied to
infinite-horizon problem with $\bG_{N-t-1,\a}$ replaced with $\bGa$ imply that
$\sa \leq \ra \leq \Sa.$ Hence, \eqref{eqn:ubdStaSa} holds.
\end{proof}

\begin{lemma}\label{lm:samevalues}
	Let Assumption~\ref{assum:qcalpha} hold for $\a\in(0,1).$ Then
	\begin{align}
		\bva (x) - \c x = \tva (x) =  \va (x) \geq 0, \qquad x\in\X.
		\label{eqn:modeltrans}
	\end{align}
	In addition, \eqref{eqn:modeltrans} implies that $\Ga (x) = \bGa (x),$ $x\in\X.$
\end{lemma}

\begin{proof}[Proof of Theorem~\ref{thm:qcall}]
	Theorem~\ref{thm:qcall} follows from Theorem~\ref{thm:qcall:t}\eqref{thm:qcall:t:2} and Lemma~\ref{lm:samevalues} because equations
	\eqref{eqn:va} and \eqref{eqn:va:trans} are equivalent, and they define the same optimal $(\sa,\Sa)$ policies.
\end{proof}

\begin{remark}\label{rm:sSoptfinite}
{\rm
	Note that $(s_{t,\a},S_{t,\a})_{t=0,1,2,\ldots,N-1}$ policies are optimal for $N=1,2,\ldots$ horizon
	inventory models with terminal costs $-\c x$ (see Theorem~\ref{thm:qcall:t}\eqref{thm:qcall:t:1}), they may not be optimal for
	finite-horizon inventory models without terminal costs (see Example~\ref{ex:sSnotoptfinite}).
	However, Theorem~\ref{thm:qcall} states that there exists
	an optimal $(\sa,\Sa)$ policy for infinite-horizon discounted cost inventory models without terminal costs.
}
\end{remark}

\begin{example}\label{ex:sSnotoptfinite}
{\rm
Consider the inventory model without terminal costs defined by the following parameters:
fixed ordering cost $K = 1,$ per unit
ordering cost $\c = 1,$ deterministic demand $D = 1,$ holding/backlog cost function $h(x) = \frac{1}{2} | x |,$
and the discount factor $\a = \frac{3}{4}.$
Since $\E[\h (x-D)] = \frac{1}{2} |x-1| + \frac{1}{4} x + \frac{3}{4},$ the function $\E[h_\a (x-D)]$ is convex
and hence quasiconvex. In addition, $\lim_{x\to -\infty} \E[\h (x-D)] = +\infty > K+\inf_{x\in\X} \E[\h (x-D)].$
Therefore, Assumption~\ref{assum:qcalpha} holds.
For the single-period problem the policy that does not order is optimal, because the cost incurred,
if nothing is ordered, is $\frac{1}{2} |x-1|$ and the cost incurred,
if $a > 0$ units are ordered, is $1+a+\frac{1}{2}|x+a-1| = 1 + \frac{1}{2}(a + |-a| + |x+a-1|) \geq 1 +
\frac{1}{2}(a + |-a + x+a-1|) > \frac{1}{2}|x-1|.$
}
\end{example}

\section{Verification of Average-Cost Optimality Assumptions for the Setup-Cost Inventory  Model}
\label{sec:mdpac}

In this section we show that, in addition to Assumption~\textbf{W*}, under
Assumption~\ref{assum:qcall}, the setup-cost inventory  model satisfies Assumption~\textbf{B} introduced by Sch\"al~\cite{Sch93}. This implies the validity of average-cost optimality inequalities (ACOIs) and the existence of stationary optimal policies;
see Feinberg~et~al.~\cite[Theorem 4]{FKZ12}. In addition, we show that, under Assumption~\ref{assum:qcall} the inventory  model satisfies the equicontinuity condition from
Feinberg and Liang~\cite[Theorem 3.2]{FLi16b}, which implies the validity of the ACOE for
the inventory  model.

As in Sch\"al~\cite{Sch93} and Feinberg~et~al.~\cite{FKZ12}, define
\begin{align}\label{defmauaw}
	\begin{split}
  		m_{\alpha}: = \underset{x\in\X}{\inf} v_{\alpha}(x), & \quad
  		u_{\alpha}(x): = v_{\alpha}(x) - m_{\alpha}, \\
  		\underline{w}: = \underset{\alpha\uparrow1}{\liminf}(1-\alpha)m_{\alpha}, & \quad
  		\bar{w}: = \underset{\alpha\uparrow1}{\limsup}(1-\alpha)m_{\alpha} .
	\end{split}
\end{align}
The function $\ua$ is called the discounted relative value function. Consider the following assumption in addition to Assumption~\textbf{W*}.

\vspace{0.1in}
\noindent
\textbf{Assumption B.}
(i) $w^{*} := \inf_{x\in\X} w^{{\rm ac}}(x)< +\infty,$ and
(ii) $\underset{\alpha\in [0,1)}{\sup} u_{\alpha}(x) <  +\infty,$ $x\in\X.$
\vspace{0.1in}

As follows from Sch\"al~\cite[Lemma~1.2(a)]{Sch93}, Assumption~\textbf{B}(i) implies that $m_\alpha<+\infty$ for all $\alpha\in [0,1).$  Thus, all the quantities in \eqref{defmauaw} are defined.  According to Feinberg et al.~\cite[Theorems~3, 4]{FKZ12}, if Assumptions~\textbf{W*} and \textbf{B} hold, then $\underline{w} = \bar{w}$ and therefore,
\begin{align}
	\lim_{\a\uparrow1} (1-\a) m_{\a} = \underline{w} = \bar{w}.
\label{eqn:bar w = underline w}
\end{align}

Define the following function on $\X$ for the sequence $\{ \a_n\uparrow 1 \}_{n=1,2,\ldots} :$
\begin{eqnarray}\label{EQN1}
  \tu(x) &:=& \liminf_{n\to +\infty,y\rightarrow x} u_{\a_n}(y) .
\end{eqnarray}
In words, ${\tilde u}(x)$ is the largest number such that ${\tilde u}(x)\le \liminf_{n\to +\infty}u_{\alpha_n}(y_n)$ for all sequences $\{y_n\to x\}.$ Since $u_{\a}(x)$ is nonnegative by definition, $\tu(x)$ is also nonnegative. The function $\tilde u,$ defined in \eqref{EQN1} for a sequence $\{ \a_n\uparrow 1 \}_{n=1,2,\ldots}$ of nonnegative discount factors, is called an average-cost relative value function.

If Assumptions~\textbf{W*} and \textbf{B} hold, then Feinberg et al.~\cite[Corollary~2]{FKZ12}
implies the validity of ACOIs and
\begin{align}\label{eqsfkz}
    w^{\phi}(x)= \underline{w} = \lim_{\alpha\uparrow 1}
    (1-\alpha)v_{\alpha} (x) = \bar{w} = w^*, \qquad x\in\X,
\end{align}
where $w^{\phi} (x)$ is defined in \eqref{eqn:sec_model def:avg cost}.
Furthermore, let us define $w:=\underline{w};$ see \eqref{eqn:bar w = underline w} and  \eqref{eqsfkz} for other equalities for $w.$

Consider the renewal counting process
\begin{align}
	\textbf{N}(t) :=  \sup \{ n=0,1,\ldots\  : \textbf{S}_n \leq t \},
	\label{renew}
\end{align}
where $t\in\Rp,$ $\textbf{S}_0 := 0,$ and
\begin{align}
	\textbf{S}_n := \sum_{j=1}^n D_j, \qquad n=1,2,\ldots\ .
	\label{eqn:Sn}
\end{align}

Observe that since $\PR(D > 0) > 0,$ $\E[\textbf{N}(t)] < +\infty,$ $t\in\Rp;$ see Resnick~\cite[Theorem 3.3.1]{Res92}.
For $x\in\X$ and $y\ge 0$  define
\begin{align}
	E_y(x) := \E[h(x - \textbf{S}_{\textbf{N}(y)+1})].
	\label{eqn:Eyx}
\end{align}
Since $x-y\leq x-\textbf{S}_{\textbf{N}(y)} \leq x$ and the function $\E[h(x-D)]$ is quasiconvex,
\begin{align}
	E_y (x) = \E[h( x-\textbf{S}_{\textbf{N}(y)} - D )] \leq \max\{ \E[h(x-y-D)],\E[h(x-D)] \} <  +\infty.
	\label{eqn:Eyxfinite}
\end{align}

\begin{theorem}\label{thm:B}
	Let Assumption~\ref{assum:qcall} hold. The inventory  model satisfies Assumption~\textbf{B}.
\end{theorem}

\begin{proof}
Assumption~\textbf{B}(i) follows from the same arguments in the first paragraph of the
proof in Feinberg and Lewis~\cite[Proposition~6.3]{FL16}.

The inf-compactness of the function $c:\X\times\A\to \R$ and the validity of Assumption~\textbf{W*}
imply that for each $\a\in [0,1)$ the function $\va $ is inf-compact (Feinberg and Lewis~\cite[Proposition 3.1(iv)]{FL07}), and therefore the set
\begin{align}
X_\alpha:=\{x\in\X :  \va (x)=m_\alpha\},
\label{eqn:Xa}
\end{align}
where $m_\a$ is defined in \eqref{defmauaw}, is nonempty and compact.
Furthermore, the validity of Assumption~\textbf{B}(i) implies that there is a compact subset $\cal K$ of $\X$ such that $\X_\alpha\subset {\cal K}$ for all $\a\in [0,1);$ see Feinberg et al.~\cite[Theorem 6]{FKZ12}.  Following  Feinberg and Lewis~\cite{FL16}, consider a bounded interval $[x^*_L,x^*_U]\subset \X$ such that
\begin{equation}\label{eq:boundmaxlul}
X_\a\subset [x^*_L,x^*_U]\qquad {\rm for\ all\ } \alpha\in [0,1).
\end{equation}
Consider an arbitrary $\alpha\in [0,1)$ and a state $x_\alpha$ such that $\va (x_\alpha)=m_\a,$ where $m_\a$ is defined in \eqref{defmauaw}.
In view of \eqref{eq:boundmaxlul}, the inequalities $x^*_L\le x_\a\le x^*_U$ hold.

Let
\begin{align}
	E(x):=\E[h(x-D)]+E_{x-x^*_L}(x)< +\infty,
	\label{eqn:Ex}
\end{align}
where the function $E_y(x)$ is defined in
\eqref{eqn:Eyx} and its finiteness is stated in \eqref{eqn:Eyxfinite}.
For $x_t=x-{\bf S}_t,$ $t=1,\ldots, {\bf N}(x-x^*_L)+1,$
\begin{align}
\E [h(x_t)]\le \E[h(x-D)]+\E [h(x-{\bf S}_{{\bf N}(x-x^*_L)+1})]= E(x), \label{est0sum}
\end{align}
where the inequality holds because the function $\E[h(x-D)]$ is quasiconvex and
$x-{\bf S}_{{\bf N}(x-x^*_L)+1} = x-{\bf S}_{{\bf N}(x-x^*_L)} - D\leq x_t = x_{t-1}-D \leq x-D$ for
$t=1,\ldots, {\bf N}(x-x^*_L)+1 .$

By considering the same policy $\sigma$ and following the arguments thereafter
as in the proof in Feinberg and Lewis~\cite[Proposition 6.3]{FL16} with the equation (6.14) there
replaced with \eqref{est0sum}, we obtain the validity of Assumption~{\bf B}.
\end{proof}

Now we establish the boundedness and the equicontinuity of the discounted relative value functions
$\ua$ defined in \eqref{defmauaw}. Consider
\begin{align}
	U(x) :=
	\begin{cases}	
K+\bar{c}(x^*_U - x), & \text{if } x < x^*_L, \\
K+\bar{c}(x^*_U - x^*_L) + (E(x) + \bar{c}\E[D])(1+\E[\textbf{N}(x-x^*_L)]), & \text{if } x\ge x^*_L ,
	\end{cases}
	\label{eqn:def U(x)}
\end{align}
where the real numbers $x^*_L$ and $x^*_U$ are defined in \eqref{eq:boundmaxlul} and the function $E(x)$ is
defined in \eqref{eqn:Ex}.

\begin{lemma}\label{lm:U(x) bounded}
Let Assumption~\ref{assum:qcall} hold.
The following inequalities hold for $\alpha\in [0,1):$
\begin{enumerate}[(i)]
	\item $\ua (x)\le U(x) < +\infty$ for all $x\in\X;$
\item If $x_*,x\in\X$ and $x_*\le x,$ then $C(x_*,x):=\sup_{y\in [x_*,x]}  U(y)< +\infty;$
\item
 $\E[U(x-D)] < +\infty$ for all $x\in\X.$
\end{enumerate}
\end{lemma}

\begin{proof}
The proof of this lemma is identical to the proof in Feinberg and Liang~\cite[Lemma 4.6]{FLi16b}.
\end{proof}

The following theorem is proved in Feinberg and Lewis~\cite[Theorem 6.10(iii)]{FL16} under
Assumption~\ref{assum:convex}. Under Assumption~\ref{assum:qcall},
since Assumptions~\textbf{W*} and \textbf{B} hold and there exist optimal $(\sa,\Sa)$ policies
for all $\a\in (\a^*,1),$ the proof of the following lemma coincides with
the proof in Feinberg and Lewis~\cite[Theorem 6.10(iii)]{FL16}.
\begin{theorem}
\label{thm:fle}
Let Assumption~\ref{assum:qcall} hold.
For each nonnegative discount factor $\alpha\in
          (\alpha^*,1)$, consider an optimal $(s^\prime_\alpha, S^\prime_\alpha)$
          policy for the discounted criterion with the discount factor $\alpha.$
          Let $\{\alpha_n\uparrow 1\}_{n=1,2,\ldots}$ be a sequence of negative numbers with $\alpha_1> \alpha^*.$  Every sequence
          $\{(s^\prime_{\alpha_n}, S^\prime_{\alpha_n})\}_{n=1,2,\ldots}$ is
          bounded, and  each its limit point $(s^*,S^*)$ defines an
          average-cost optimal $(s^*,S^*)$ policy. Furthermore, this policy satisfies
          the optimality inequality
\begin{align}
	w + \tu (x) \geq \min \{ \min_{a\geq 0}[ K + H(x+a)], H(x) \} - \bar{c}x,
	\label{eqn:IC ACOE-1}
\end{align}
where
\begin{align}
	H(x) := \bar{c}x +\E[h(x-D)] + \E[\tu (x-D)],
	\label{eqn:IC ACOE-2}
\end{align}
where the function $\tilde u$ is defined in \eqref{EQN1} for an arbitrary subsequence
         $\{\alpha_{n_k}\}_{k=1,2,\ldots}$ of $\{\alpha_{n}\}_{n=1,2,\ldots}$ satisfying $(s^*,S^*)=\lim_{k\to +\infty} (s^\prime_{\alpha_{n_k}}, S^\prime_{\alpha_{n_k}}).$
\end{theorem}

Recall the following definition of equicontinuity.
\begin{definition}
	A family $\mathcal{H}$ of real-valued functions on a metric space $X$
	is called equicontinuous at the point $x\in X$ if for each $\eps > 0$
	there exists an open set $G$ containing $x$ such that
	\begin{equation*}
		| h(y) - h(x) | < \eps  \;\;\; \text {for all}\  y\in G \text{ and for all } h\in\mathcal{H}.
	\end{equation*}
	The family of functions  $\mathcal{H}$ is called equicontinuous (on $X$) if it is equicontinuous at
	all $x \in X.$
	\label{def:equicontinuous}
\end{definition}

Consider the following assumption on the discounted relative value functions.

\vspace{0.1in}
\noindent
\textbf{Assumption EC} (Feinberg and Liang~\cite{FLi16b})\textbf{.} There exists a sequence $\{ \a_n\uparrow1 \}_{n=1,2,\ldots}$ of nonnegative discount factors such that

(i) the family of functions $\{u_{\a_n}\}_{n=1,2,\ldots}$ is equicontinuous, and

(ii) there exists a nonnegative measurable function $U(x),$ $x\in\X,$ such that
$U(x)\geq u_{\a_n}(x),$ $n=1,2,\ldots,$ and $\int_{\X} U(y)q(dy|x,a) <  +\infty$ for all $x\in\X$ and $a\in \A.$
\vspace{0.1in}

The following theorem provides sufficient conditions under which there exist a stationary policy $\phi$ and a function $\tu(\cdot)$ such that the ACOEs hold for MDPs satisfying
Assumptions~\textbf{W*} and \textbf{B}.

\begin{theorem}[{Feinberg and Liang~\cite[Theorem 3.2]{FLi16b}}]
Let Assumptions~\textbf{W*} and \textbf{B} hold. Consider a sequence $\{ \a_n\uparrow1 \}_{n=1,2,\ldots}$ of nonnegative discount factors.
If Assumption \textbf{EC} is satisfied for the sequence $\{ \a_n\}_{n=1,2,\ldots},$
then the following statements hold.
\begin{enumerate}[(i)]
	\item There exists a subsequence $\{ \a_{n_k} \}_{k=1,2,\ldots}$ of $\{ \a_n \}_{n=1,2,\ldots}$ such that $\{u_{\a_{n_k}}(x)\}$ converges pointwise to $\tu(x),$ $x\in\X,$ where $\tu (x)$ is defined in \eqref{EQN1} for the subsequence $\{ \a_{n_k}\}_{k=1,2,\ldots},$ and the convergence is uniform on each compact subset of $\X.$ In addition, the function $\tu(x)$ is continuous. \label{stat:1}
	\item There exists a stationary  policy $\phi$ satisfying the ACOE with the nonnegative function $\tilde u$ defined for the sequence  $\{ \a_{n_k} \}_{k=1,2,\ldots}$ mentioned in statement \eqref{stat:1}, that is, for all $x\in\X,$
     \begin{align}
  		w + \tu(x) = c(x,\phi(x)) + \int_{\X} \tu(y)q(dy|x,\phi(x)) =
  		\min_{a\in \A} [ c(x,a) + \int_{\X} \tu(y)q(dy|x,a) ],
  	\label{EQN11}
	\end{align}
and every stationary policy satisfying \eqref{EQN11} is average-cost optimal.
\end{enumerate}
\label{thm:acoe}
\end{theorem}

The following theorem shows that the equicontinuity conditions stated in Theorem~\ref{thm:acoe} holds for the inventory  model with holding/backlog costs satisfying quasiconvexity assumptions.

\begin{theorem}\label{thm:u a equicont}
Let Assumption~\ref{assum:qcall} hold. Then for each $\b \in (\a^*, 1),$
the family of functions $\{u_{\a}\}_{\a\in [\b,1)}$ is equicontinuous on $\X.$
\end{theorem}

\begin{proof}
We first follow the same procedure as in the proof in Feinberg and Liang~\cite[Theorem~4.9(a)]{FLi16b} to prove that, for each $\b\in (\a^*,1),$ the family of functions
$\{u_{\a}\}_{\a\in[\b,1)}$ is equicontinuous on $(-\infty,M]$ for any given $M.$

According to Theorem~\ref{thm:qcall}, since the $(\sa,\Sa)$ policies are optimal, the arguments provided to
prove Equation (4.38) in the proof in Feinberg and Liang~\cite[Theorem~4.9(a)]{FLi16b} imply that
there exist constants $b>0$ and $\delta_0 > 0$ such that, for each $\b\in(\a^*,1),$ $s_{\a}\in (-b,b)$ and
\begin{align}
	-b\leq s_{\a} - \delta_0 < s_{\a} + \delta_0\leq b , \qquad \a\in [\b,1) .
	\label{eqn:sab}	
\end{align}

For each $\b\in (\a^*,1),$ let $\a\in [\b,1).$ Consider $M > b,$ $z_1$ and $z_2$ satisfying $M > z_1,z_2\geq s_{\a}.$ Without loss of generality, assume that $z_1<z_2.$
With $\a_n$ replaced with $\a,$ the arguments provided
before Equation (4.20) in the proof in Feinberg and Liang~\cite[Lemma 4.7]{FLi16b} imply that
\begin{align}
\begin{split}
	& |u_{\a}(z_1) - u_{\a}(z_2)| = |v_{\a}(z_1) - v_{\a}(z_2)|  \\
	= & {\Big |}\E[\sum_{j=1}^{\textbf{N}(z_1-s_{\a})+1} \a^{j-1}(\tilde{h}(z_1-\textbf{S}_{j-1}) - \tilde{h}(z_2-\textbf{S}_{j-1}))  \\
	& + \a^{\textbf{N}(z_1-s_{\a})+1} (v_{\a}(z_1-\textbf{S}_{\textbf{N}(z_1-s_{\a})+1}) - v_{\a}(z_2-\textbf{S}_{\textbf{N}(z_1-s_{\a})+1}) )] {\Big |}  \\
	\leq & \E[\sum_{j=1}^{\textbf{N}(M+b)+1} |\tilde{h}(z_1-\textbf{S}_{j-1}) - \tilde{h}(z_2-\textbf{S}_{j-1})|]   \\
	& +  \E[ |u_{\a}(z_1-\textbf{S}_{\textbf{N}(z_1-s_{\a})+1}) - u_{\a}(z_2-\textbf{S}_{\textbf{N}(z_1-s_{\a})+1})|] ,
	\end{split}
	\label{eqn:ubd2}
\end{align}
where $\tilde{h} (x) := \E[h(x-D)],$  the inequality holds because of $\a < 1,$ in view of the standard properties of expectations and absolute values,
and because $-b < s_{\a}\leq z_1 < M.$ Recall that the function $\tilde{h} (x)$ is continuous
and finite. Therefore, the function $\tilde{h}$ is uniformly continuous on the closed interval $[-(M+2b), M].$
In addition, Assumption~\ref{assum:qcall} implies that
the function $\tilde{h}$ is quasiconvex.

Since $-(M+2b) < x-\textbf{S}_{j-1} < M$ for all $x\in [-b,M]$ and $j=1,2,\ldots,\textbf{N}(M+b)+1,$
the nonnegativity and quasiconvexity of $\tilde{h}$ imply that for all $x\in (-b,M)$
\begin{align}
	0\leq \tilde{h}(x-\textbf{S}_{j-1}) \leq \max \{ \tilde{h}(-(M+2b)),\tilde{h}(M) \}.
	\label{eqn:th1}
\end{align}
Furthermore, for $-b < z_1<z_2 < M,$
\begin{align}
\begin{split}
	&\E[\sum_{j=1}^{\textbf{N}(M+b)+1} |\tilde{h}(z_1-\textbf{S}_{j-1}) - \tilde{h}(z_2-\textbf{S}_{j-1}) |]
	 \\
	\leq &\E[\sum_{j=1}^{\textbf{N}(M+b)+1} \max \{ \tilde{h}(-(M+2b)),\tilde{h}(M) \}]
	 \\
	\leq &\E[\textbf{N}(M+b)+1]\max \{ \tilde{h}(-(M+2b)),\tilde{h}(M) \} <  +\infty,
	\end{split}
	\label{eqn:thfinite}
\end{align}
where the first inequality follows from \eqref{eqn:th1}, the second one follows from Wald's identity,
and the last one follows from the finiteness of the function $\tilde{h}.$
Therefore, for $-b < z_1<z_2 < M,$
\begin{align}
\begin{split}
	& \lim_{z_1\to z_2} \E[\sum_{j=1}^{\textbf{N}(M+b)+1} |\tilde{h}(z_1-\textbf{S}_{j-1}) - \tilde{h}(z_2-\textbf{S}_{j-1}) |] \\
	= & \E[\sum_{j=1}^{\textbf{N}(M+b)+1} \lim_{z_1\to z_2} |\tilde{h}(z_1-\textbf{S}_{j-1}) - \tilde{h}(z_2-\textbf{S}_{j-1}) |] = 0,
\end{split}
	\label{eqn:thlimit}
\end{align}
where the first equality follows from \eqref{eqn:thfinite} and Lebesgue's dominated convergence theorem,
and the second one follows from the continuity of $\tilde{h}.$

Consider $\eps > 0.$ In view of \eqref{eqn:thlimit}, since $-b < z_1<z_2<M,$ there exists
$\delta_1 \in (0,\delta_0)$ such that for $s_{\a}\leq z_1<z_2$ satisfying $|z_1 - z_2|<\delta_1$
\begin{align}
	\E[\sum_{j=1}^{\textbf{N}(M+b)+1} |\tilde{h}(z_1-\textbf{S}_{j-1}) - \tilde{h}(z_2-\textbf{S}_{j-1})|] 	
	\leq  \frac{\eps}{2}.
	\label{eqn:ubd2-1}
\end{align}

Additional arguments are needed to estimate the last term in \eqref{eqn:ubd2}.
Next we prove that there exists $\delta_2\in (0,\delta_1)$ such that for $x\in [\sa,\sa+\delta_2],$
\begin{align}
	|\ua (x) - \ua (\sa)| < \frac{\epsilon}{4}.
	\label{eqn:ubd ux uan}
\end{align}

Let $x\ge s_\a.$ Then
\begin{equation}\va (x ) =
		\tilde{h}(x) + \a \E[\va (x-D)] \label{firstfrom423}
\end{equation}
and
\begin{align}
\begin{split}
&\E[\va(x-D)] =
\PR(D\geq x - \sa) \E[\bar{c}(\sa - x + D)|D\geq x - \sa]  \\
&+ \PR(0<D< x - \sa) \E[\va (x-D)|0<D< x - \sa]
+ \PR(D = 0) \va (x)
\end{split}
\label{secondfrom423}
\end{align}
Formulas \eqref{firstfrom423} and \eqref{secondfrom423} imply
\begin{align}
[1&-\a \PR(D=0)]	\va (x)  =  \tilde{h}(x) + \a (\PR(D\geq x - \sa) \E[\bar{c}(\sa - x + D)|D\geq x - \sa] \nonumber \\
& + \PR(0<D< x - \sa) \E[\va (x-D)|0<D< x - \sa]).
	\label{eqn: ua-1}
\end{align}
 Therefore, since $\ua (y_1) - \ua (y_2) = \va (y_1) - \va (y_2)   $ for all $y_1,y_2\in\X,$ for $x\in [\sa,\sa+\delta_1]$
\begin{align}
\begin{split}
	& [1-\a\PR(D=0)]|\ua (x) - \ua (\sa)| = [1-\a\PR(D=0)]|\va (x) - \va (\sa)|  \\
	&=   \Big |\tilde{h}(x)-\tilde{h}(\sa) +\a\PR(D\geq x - \sa)\bar{c}(\sa - x)  \\ & +\a\PR(0<D< x - \sa)\E[\ua (x-D)
	-\ua (\sa-D)|0<D< x - \sa]   \Big |  \\
	&\leq  |\tilde{h}(x)-\tilde{h}(\sa)|+\bar{c}(x-\sa)+ 2 \PR(0<D< x-\sa) C(-b,b),
\end{split}
\label{eqn:dif ua ubd}
\end{align}
where the nonnegative function $C$ is defined in Lemma~\ref{lm:U(x) bounded}. Let us define $Q_1:= (1-\PR(D=0))^{-1},$ and $Q_2(x,\sa ):=\PR(0<D< x-\sa).$   Recall that $\PR(D>0)>0, $ which is equivalent to $\PR(D=0)<1.$ Since $(1-\a\PR(D=0))^{-1}\le Q_1,$  formula \eqref{eqn:dif ua ubd} implies that
\begin{equation}
|\ua (x) - \ua (\sa)|\le Q_1(|\tilde{h}(x)-\tilde{h}(\sa)|+\bar{c}(x-\sa) + 2 Q_2(x,\sa ) C(-b,b)).
\end{equation}
Since the function $\tilde{h}$ is uniform continuous on the interval $[-(M+2b),M],$ all three summands in the right-hand side of the last equations converge uniformly in $\a$ to 0 as $x\downarrow \sa.$  Therefore, there exists $\delta_2\in (0,\delta_1)$ such that \eqref{eqn:ubd ux uan} holds for all $x\in [\sa,\sa+\delta_2].$

Since $\ua  (x) = \bar{c}(\sa - x)  + \ua  (\sa)$ for all $x\leq \sa,$ then for all $x,y\le \sa$
\begin{align}
	|\ua (x) - \ua (y)| =\bar{c}|x-y|< \frac{\epsilon}{4},
	\label{eqn:ubd u x < san}
\end{align}
for $|x-y|<\frac{\epsilon}{4{\bar c}}.$ Let $\delta_3:=\min\{\frac{\epsilon}{4{\bar c}},\delta_2\}.$ Then \eqref{eqn:ubd u x < san} holds for $|x-y|<\delta_3.$

For $x \leq \sa \leq y$ satisfying $|x-y|< \delta_3$
\begin{align}
	|\ua (x) - \ua (y)| \leq |\ua (x) - \ua (\sa)| + |\ua (\sa) - \ua (y)|
	< \frac{\epsilon}{2},
	\label{eqn:ubd x < san < y}
\end{align}
where the first inequality is the triangle property and the second one follows from \eqref{eqn:ubd ux uan} and \eqref{eqn:ubd u x < san}.
Therefore, \eqref{eqn:ubd ux uan}, \eqref{eqn:ubd u x < san} and \eqref{eqn:ubd x < san < y} imply that $|\ua (x) - \ua (y)| <\frac{\epsilon}{2}$ for all $x,y \leq \sa + \delta_3$ satisfying $|x-y|< \delta_3.$
Then for $|z_1-z_2|< \delta_3$ with probability $1$
\[
|\ua (z_1-\textbf{S}_{\textbf{N}(z_1-\sa) + 1}) - \ua (z_2-\textbf{S}_{\textbf{N}(z_1-\sa) + 1})|  < \frac{\epsilon}{2},
\]
and therefore
\begin{align}
	 \E[|\ua (z_1-\textbf{S}_{\textbf{N}(z_1-\sa) + 1}) - \ua (z_2-\textbf{S}_{\textbf{N}(z_1-\sa) + 1})|]  < \frac{\epsilon}{2},
	\label{eqn:ubd2-2}
\end{align}
Formulae \eqref{eqn:ubd2}, \eqref{eqn:ubd2-1}. and  \eqref{eqn:ubd2-2} imply that for $z_1,$ $z_2\geq \sa$ satisfying $|z_1-z_2|< \delta_3$
\begin{align}
	|\ua (z_1) - \ua (z_2) | < \epsilon .
	\label{eqn:ubd x12>san}
\end{align}

Therefore, \eqref{eqn:ubd u x < san}, \eqref{eqn:ubd x < san < y}, and \eqref{eqn:ubd x12>san} imply that
for all $x,$ $y \in (-\infty, M)$ satisfying $|x-y|<\delta_3$
\begin{align*}
	|u_{\a}(x) - u_{\a}(y) | < \eps .
\end{align*}
Since $M$ can be chosen arbitrarily large, for each $\b\in (\a^*,1)$ the family
of functions $\{u_{\a}\}_{\a\in[\b,1)}$ is equicontinuous on $\X.$
\end{proof}

\begin{theorem}\label{thm:Gacont}
	Let Assumption~\ref{assum:qcall} hold. Then for $\a \in (\a^*, 1),$ the functions $\va$ and
	$\Ga$ are continuous on $\X.$
\end{theorem}

\begin{proof}[Proof of Theorem~\ref{thm:Gacont}]
	According to Theorem~\ref{thm:qcall}, there exists an optimal $(\sa,\Sa)$ optimal policy for the infinite-horizon problem. In addition, Theorem~\ref{thm:u a equicont} implies that the function $\va (x) = \ua (x) + \ma$ is continuous on $\X.$ Therefore, since the function $\E[h(x-D)]$ is continuous, the same arguments in the proof of Feinberg and Liang~\cite[Theorem 5.3]{FLi16a} starting from the definition of the function $g_\a$ there imply that the function $\Ga$ is continuous on $\X.$
\end{proof}

\section{Setup-Cost Inventory  Model: Average Costs per Unit Time}
\label{sec:acoe}

The following theorem establishes the convergence of discounted-cost optimality equations to the ACOEs for the
described inventory model and the optimality of $(s,S)$ policies under the average cost criterion under
Assumption~\ref{assum:qcall}. It is proved in Chen and Simchi-Levi~\cite{CS04b} that there exists an average-cost
optimal $(s,S)$ policy if only Assumption~\ref{assum:qcalpha} for $\a = 1$ is assumed. We are interested in approximating the average-cost optimal $(s,S)$ policy from the discount-cost optimal $(\sa,\Sa)$ policies as the discount factor $\a\uparrow1.$

\begin{theorem}\label{thm:inventory:acoe}
Let Assumption~\ref{assum:qcall} hold.
For every sequence $\{\alpha_n\uparrow 1\}_{n=1,2,\ldots}$ of nonnegative discount factors with
$\alpha_1 > \a^*,$ there exist a subsequence $\{\alpha_{n_k}\}_{k=1,2,\ldots}$
of $\{\alpha_n\}_{n=1,2,\ldots},$ a stationary policy $\varphi,$ and a function $\tu$ defined in
\eqref{EQN1} for the subsequence $\{\alpha_{n_k}\}_{k=1,2,\ldots}$ such that for all $x\in\X$
\begin{align}\label{eqn:IC ACOE}
\begin{split}
  w + \tu(x)  &= K \indF_{\{\varphi(x)>0\}} + H(x+\varphi(x)) - \bar{c}x \\
  &=\min \{ \min_{a\geq 0}[ K + H(x+a)], H(x) \} - \bar{c}x,
  \end{split}
\end{align}
where the function $H$ is defined in \eqref{eqn:IC ACOE-2}. In addition, the functions $\tilde u$ and $H$ are
continuous and inf-compact, and a stationary optimal policy $\varphi$ satisfying \eqref{eqn:IC ACOE} can be selected as an $(s^*,S^*)$ policy described in Theorem~\ref{thm:fle}.  It also can be selected  as an $(s,S)$ policy with the real numbers $S$ and $s$ satisfying \eqref{eqn:def S} and  defined in   \eqref{eqn:def s} respectively for $f(x)=H(x),$ $x\in\X.$
\end{theorem}

\begin{remark}
{\rm
The relations between the function $\tu$ in the ACOE \eqref{eqn:IC ACOE} and the solutions to the ACOE constructed by Chen and Simchi-Levi~\cite{CS04b} are currently not clear.
}
\end{remark}

To prove Theorem \ref{thm:inventory:acoe}, we first establish several properties of the
average-cost relative value function. The proofs of the lemma and corollary presented in this section
are available in Appendix~\ref{sec:acoe:apx}.

\begin{lemma}\label{lm:Horder}
Let Assumption~\ref{assum:qcall} hold. Consider the function $\tu$
defined in \eqref{EQN1} for a sequence $\{ \a_n \}_{n=1,2,\ldots}$ such that $\alpha_n\uparrow 1$
and $\a_1>\a^*.$
Then the following statements hold:
\begin{enumerate}[(i)]
	\item For $x\leq y$
\begin{align}
	\tu  (x) + \c x  & \leq \tu  (y) + \c y + K,  \label{eqn:tu1} \\
	H(y) - H (x) & \geq \E[h(y-D)] - \E[h(x-D)] - \a K. \label{eqn:H1}
\end{align}
	\item For $x\leq y \leq \rone$
\begin{align}
	\tu (y) + \c y - \tu (x) - \c x & \leq 0,   \label{eqn:tude1} \\
	H (y) - H (x) & \leq  0. \label{eqn:Hde1}
\end{align}
\end{enumerate}
\end{lemma}

\begin{proof}[Proof of Theorem~\ref{thm:inventory:acoe}]
The proof of this thoerem is identical to the proof in Feinberg and Liang~\cite[Theorem 4.5]{FLi16b}
with
(i) Lemmas~4.6 and 4.7 there replaced with Lemma~\ref{lm:U(x) bounded} and
Theorem~\ref{thm:u a equicont} from this paper respectively;
and (ii) the proof of the $K$-convexity of the functions $u$ and $H$ and the optimality of $(s,S)$
policies under the average cost criterion replaced with the following arguments.
Consider the cases (1)-(3) in the proof of Theorem~\ref{thm:qcall:t}\eqref{thm:qcall:t:1} with
$G_{N-t-1,\a},$ $\h$ and $\ra$ replaced with $H,$ $h,$ and $\rone$ respectively. Then
Lemma~\ref{lm:Horder} implies that there exists an optimal $(s,S)$ policy,
with the real numbers $S$ and $s$ satisfying \eqref{eqn:def S} and  defined in
\eqref{eqn:def s} for $f:=H.$
\end{proof}

Furthermore, the continuity of average-cost relative value functions implies the following corollary.

\begin{corollary}\label{cor:ordering at sa}
Let Assumption~\ref{assum:qcall} hold, the state space $\X = \R,$ and the action space
$\A = \Rp.$
For the $(s,S)$ policy defined in Theorems~\ref{thm:inventory:acoe}, consider the stationary policy $\varphi$ coinciding with this policy at all $x\in\X,$ except $x=s,$ and with   $\varphi(s)=S-s.$  Then the stationary policy $\varphi$ also satisfies the optimality equation~\eqref{eqn:IC ACOE}, and is therefore average-cost optimal.
\end{corollary}

\section{Convergence of Optimal Lower Thresholds $\sa$}
\label{sec:cvg sa}

This section  establishes for the inventory  model with holding/backlog costs satisfying quasiconvexity assumptions the convergence of discounted optimal lower thresholds $\sa\to s$ as $\a\uparrow1,$ where $s$ the average-cost optimal lower threshold (stated in Theorem~\ref{thm:inventory:acoe}). In this and the following sections, we assume that the state space $\X=\R$ and the action
sets $\A = A(x) = \Rp$ for all $x\in\X.$ This means an arbitrary nonnegative amount of inventory can be ordered at any state.

The quasiconvexity of $\E[h(x-D)]$ assumed in Assumption~\ref{assum:qcall} implies that the function
$\E[h(x-D)]$ is nonincreasing on $(-\infty,\rone),$ where $\rone$ is defined in \eqref{eqn:def r alpha r*}.
The stronger Assumption~\ref{assum:decrease} is used in this section and Section~\ref{sec:cvg ua}.

The following theorem establishes the convergence of the discounted optimal lower thresholds
$\sa$ when the discount factor $\a$ converges to $1.$

\begin{theorem}\label{thm:limit s alpha}
	Let Assumptions~\ref{assum:qcall} and \ref{assum:decrease} hold for $\a = 1.$
	Then the limit
	\begin{align}
		s^* := \lim_{\a\uparrow 1} \sa
		\label{eqn:def s*}
	\end{align}
	exists and $s^* \leq \rone,$ where $\rone$ is defined in \eqref{eqn:def r alpha r*}.
\end{theorem}

\begin{remark}\label{rm:s}
{\rm
As shown in Corollary~\ref{cor:inventory:acoe}, if Assumptions~\ref{assum:qcall} and
\ref{assum:decrease} hold for $\a=1,$  then all the sequences $\{\a_n\uparrow1\}_{n=1,2,\ldots}$ define
the same functions $\tilde{u}$ and $H$ in \eqref{eqn:IC ACOE}, and, according to Theorem~\ref{thm:uniqueac},
there exists a unique threshold $s^*,$ for which there is an $(s^*,S^*)$ policy satisfying the ACOE \eqref{eqn:IC ACOE}.
}
\end{remark}

Before the proof of Theorem~\ref{thm:limit s alpha}, we first state several auxiliary facts.
Consider the infinite-horizon value function $\bva$ for the model with zero unit and terminal costs. According to
Lemma~\ref{lm:samevalues}, $\bva (x) - \c x = \va (x),$ $x\in\X.$ For $x\in\X,$ define
\begin{align}
\bma := \min_{x\in\X} \{ \bva (x) \} \quad \text{and} \quad \bua (x) := \bva (x) - \bma.
\label{eqn:def bua}
\end{align}

If there exists an $\a$-discount optimal $(\sa,\Sa)$ policy, then \eqref{eqn:va:trans} can be written as
\begin{align}
	\bva (x) =
	\begin{cases}
		\Ga (x)  				& \text{if } x \geq \sa ,  \\
		K + \Ga (\Sa)  	& \text{if } x < \sa ,  \\
	\end{cases}
	\label{eqn:bva sS}
\end{align}
which implies that
\begin{align}
	\bma = \min_{x\in\X} \{ \Ga (x) \} = \Ga (\Sa).
	\label{eqn:bma}
\end{align}

Consider $x_{\a} \in X_{\a},$ where $X_{\a}$ is defined in \eqref{eq:boundmaxlul}. For $\a\in[0,1)$
\begin{align}
	\bma \leq \bva (x_{\a}) = \ma + \c x_{\a} \leq \ma + \c x_U^* ,
	\label{eqn:ubd bma}
\end{align}
where  $x^*_U$ is defined in \eqref{eq:boundmaxlul}.
In view of \eqref{eqn:bva sS}, the continuity of $\bva(x)$ implies that $\bva(x) = \bva(\sa)$ for all $x \leq \sa.$
Therefore,
\begin{align}
	\bma = \inf_{x\geq \sa} \bva(x) = \inf_{x\geq \sa} \{ \va (x) + \c x \}  \geq \inf_{x\geq \sa} \{ \va (x) + \c \sa \} \geq \ma + \c \sa,
	\label{eqn:lbd bma}
\end{align}
where the first inequality holds because $x\geq \sa$ and the last one follows from
$\ma = \inf_x \va (x).$
Then \eqref{eqn:ubd bma} and \eqref{eqn:lbd bma} imply
\begin{align}
	\ma + \c \sa \leq \bma \leq \ma + \c x_U^*.
	\label{eqn:bounds on bma}
\end{align}

For $\a\in (\a^*,1)$ define the set of all possible optimal discounted lower thresholds
\begin{align}
	\mathcal{G}_\alpha := \{ x \leq \Sa : \Ga (y) = K + \Ga (\Sa) \text{ for all } y\in[\sa,x] \},
	\label{eqn:set of sa}
\end{align}
where $\Sa$ satisfies \eqref{eqn:def S} and $\sa$ is defined in \eqref{eqn:def s} with $f:=\Ga.$
Note that $\sa\in \mathcal{G}_\a$ and $y \geq \sa$ for all $y\in\mathcal{G}_\a.$

\begin{remark}\label{rm:mathcalGa}
{\rm
The set $\mathcal{G}_\a$ is not empty if $\X=\R$ because the function $\Ga $ is continuous (see Theorem~\ref{thm:Gacont}) and $\lim_{x\to-\infty} \Ga (x) > K + \Ga (\Sa).$ If $\X = \Z,$ as this takes place for problems with discrete commodity, it is possible that $\mathcal{G}_\a$ is an empty set.
}
\end{remark}

The following three lemmas state the relations between parameters defined in this section. The proofs of
the lemmas presented in this section are available in Appendix~\ref{sec:cvg sa:apx}.

\begin{lemma}\label{lm:sa}
Let Assumption~\ref{assum:qcall} hold.
Then, for all $\a\in (\a^*,1)$ and $y \in \mathcal{G}_\a,$
	\begin{align}
		(1-\a)(\bma + K) = \E[\h (y - D)].
		\label{eqn:value and ha}
	\end{align}
\end{lemma}

\begin{lemma}\label{lm:r alpha <= r*}
Let Assumption~\ref{assum:qcall} hold.
	Then, for all $\a\in (\a^*,1)$ and $y \in \mathcal{G}_\a,$
	\begin{align*}
		y \leq \ra \leq \rone.
	\end{align*}
\end{lemma}

\begin{lemma}\label{lm:limit m*}
Let Assumption~\ref{assum:qcall} hold. Then
	\begin{align}
		\lim_{\a\uparrow 1} (1-\a) \bma = \lim_{\a\uparrow1}\E[h_{\a} (s_{\a} - D)] = w .
		\label{eqn:equality for limit}
	\end{align}
\end{lemma}

\begin{proof}[Proof of Theorem~\ref{thm:limit s alpha}]
The proof is by contradiction.
	According to Theorem~\ref{thm:fle}, for $\a_n\uparrow 1,$ $n=1,2,\ldots,$
with $\a_1 > \a^*,$ every sequence $\{(s_{\a_n},S_{\a_n})\}_{n=1,2,\ldots}$ is bounded.
Consider two real numbers $s^{(1)} < s^{(2)}$ such
that there exist two sequences $\{\a_n\}_{n=1,2,\ldots}$ and $\{\tilde{\a}_n\}_{n=1,2,\ldots}$ satisfying
$\lim_{n\to +\infty} s_{\a_n} = s^{(1)}$ and $\lim_{n\to +\infty} s_{\tilde{\a}_n} = s^{(2)}.$

Since the function $\E[h(x-D)]$ is continuous,
\begin{align}
	\lim_{n\to +\infty}\E[h(s_{\a_n} - D)] = \E[h(s^{(1)} - D)].
	\label{eqn:limit:part 1}
\end{align}
Therefore,
\begin{align}
	\lim_{n\to +\infty} \E[h_{\a_n} (s_{\a_n} - D)] &= \lim_{n\to +\infty} \big\{ \E[h (s_{\a_n} - D)] + (1-\a_n)\c s_{\a_n} + \a_n \c \E[D] \big \} \nonumber \\
	& = \E[h(s^{(1)} - D)] + \c \E[D],
	\label{eqn:limit ha(sn-D)}
\end{align}
where the second equality follows from \eqref{eqn:limit:part 1} and $s_{\a_n}\to s^{(1)}\in\R$ as $\a_n\uparrow 1.$ According to Lemma~\ref{lm:limit m*}, $\E[h(s^{(1)} - D)] = w - \c\E[D].$ By the same arguments with $\a_n$ replaced with $\tilde{\a}_n,$ $\E[h(s^{(2)} - D)] = w - \c\E[D].$ Therefore,
\begin{align}
	\E[h(s^{(1)} - D)] = \E[h(s^{(2)} - D)].
	\label{eqn:h(s1)=h(s2)}
\end{align}

According to Lemma~\ref{lm:r alpha <= r*}, $\sa \leq \ra$ for all $\a\in (\a^*,1).$ Therefore
\begin{align}
	s^{(2)} = \lim_{n\to +\infty} s_{\tilde{\a}_n} \leq \liminf_{n\to +\infty} \rtan \leq \rone,
	\label{eqn:s<r*}
\end{align}
where the last inequality follows from Lemma~\ref{lm:r alpha <= r*}. Since $s^{(1)} < s^{(2)} \leq \rone,$
Assumption~\ref{assum:decrease} implies that
\begin{align}
	\E[h(s^{(1)} - D)] > \E[h(s^{(2)} - D)],
	\label{eqn:h(s1)>h(s2)}
\end{align}
which contradicts \eqref{eqn:h(s1)=h(s2)}. Thus, the limit $\lim_{\a\uparrow 1} \sa$ exists and
\eqref{eqn:s<r*} implies that $s^* \leq \rone.$
\end{proof}

The following theorem establishes the uniqueness of possible optimal lower thresholds
for the inventory  model with convex cost functions under the discounted criterion.

\begin{theorem}\label{thm:unique}
Let Assumptions~\ref{assum:qcall} and \ref{assum:decrease} hold for $\a\in(\a^*,1).$ Then
\begin{align*}
	\mathcal{G}_\a = \{ \sa \},
\end{align*}
where $\mathcal{G}_\a$ and $\sa$ are defined in \eqref{eqn:set of sa} and
\eqref{eqn:def s} with $f:=\Ga$ respectively.
\end{theorem}

\begin{proof}
Recall that $\sa\in\mathcal{G}_\a$ and $y\geq \sa$ for all $y\in\mathcal{G}_\a.$
The proof is by contradiction.
Assume that there exists $y_1\in \mathcal{G}_\a$ such that $y_1 > \sa.$
According to Lemma~\ref{lm:sa},
\begin{align*}
	\E[\h (y_1-D)]  = (1-\a)(\bma + K) = \E[\h (\sa - D)].
\end{align*}
Since Assumption~\ref{assum:decrease} holds for the discount factor $\a,$ $\ra< y_1,$ where $\ra$ is defined in \eqref{eqn:def r alpha r*}. However, according to
Lemma~\ref{lm:r alpha <= r*}, $y_1\leq \ra ,$ which implies that $y_1\leq \ra <y_1.$
Therefore, $\mathcal{G}_\a = \{ \sa \}.$
\end{proof}

\section{Convergence of Discounted Relative Value Functions}
\label{sec:cvg ua}

This section establishes the convergence of discounted relative value functions to the average-cost
relative value function for the setup-cost inventory  model when the discount factor tends to $1.$
This is a stronger result than the convergence for a subsequence that follows from Theorem~\ref{thm:inventory:acoe}.
We recall that in this section it is assumed that $\X = \R$ and $\A = A(x) = \Rp$ for all $x\in\X.$
The proofs of the proposition and lemmas presented in this section are available in Appendix~\ref{sec:cvg ua:apx}.

Let us define
\begin{align}
	u (x) := \liminf_{\a\uparrow1, y\to x} u_{\a} (y).
	\label{eqn:FKZu}
\end{align}
According to Feinberg~et~el.~\cite[Theorems~3,~4]{FKZ12}, the ACOI holds for the relative
value functions $\tilde{u}$ and $u$ defined in \eqref{EQN1} and \eqref{eqn:FKZu} respectively.

The following theorem states the convergence of discounted relative value functions,
when the discount factor converges to 1, to the average-cost relative value function $u.$

\begin{theorem}\label{thm:ua converge}
Let Assumptions~\ref{assum:qcall} and \ref{assum:decrease} hold for $\a=1.$ Then,
	\begin{align}
		\lim_{\a\uparrow1} u_{\a}(x) = u (x), \qquad x\in\X,
	\label{eqn:ua converge-1}
\end{align}
and the function $u$ is continuous.
\end{theorem}

In particular, Theorem~\ref{thm:ua converge} implies that the function $\tilde{u}$ defined in \eqref{EQN1}
is the same for every particular sequence $\{\a_n\uparrow1\}_{n=1,2,\ldots}.$
The following example demonstrates that this is not true in general under Assumptions~\textbf{W*} and~\textbf{B}.

\begin{example}\label{ex:mdplimit}
{\rm
Consider a MDP with state space $\X = \{-2, -1, 0, 1, 2, \ldots \}$ and action space
$\A = \{a^s , a^c\},$ where the action $a^s$ stands for ``stop'' and the action $c$ stands for ``continue'';
see Figure~\ref{fig:mdplimit}.
Let $A(-1) = \A$ and $A(n) = \{ a^c \}$ for $n \in \X\setminus \{ -1 \}.$ The transition probabilities are
$P(-1 | -1, a^s) = 1$ and $P(n + 1 | n, a^c) = 1$ for $n\in \X.$
The costs are $c(-2, a^c) = 0,$ $c(-1, a) = 1$ for $a\in\A,$
and $c(n, a^c) = \zone_n $ for $n = 0,1,\ldots,$ where $\zone_n$ is defined as
\begin{align}
	\zone_n =
	\begin{cases}
		z_0 + 1, & \text{if } n = 0, \\
		z_n - z_{n-1} + 1, & \text{if } n = 1,2,\ldots,
	\end{cases}
	\label{eqn:ex1:defzn1}
\end{align}
where the sequence $z_n$ is taken from Bishop et al.~\cite[Equation (11)]{BFZ14}:
\begin{align*}
	z_n =
	\begin{cases}
		1, & \text{if } D(2k - 1) \leq n < D(2k), \quad k = 1,2,\ldots, \\
		0, & \text{otherwise,}
	\end{cases}
\end{align*}
where $D(k) := \sum_{i=1}^k i! ,$ $k = 1,2,\ldots\ .$ For the sequence $\{ z_n \}_{n=0,1,\ldots},$ define the function
\begin{align}
	f(\a) := (1 - \a) \sum_{i=0}^{\infty} z_i \a^i , \qquad \a \in [0,1) .
	\label{eqn:ex1:deffalpha}
\end{align}
\begin{figure}[ht]
  \centering
  \caption{MDP described in Example~\ref{ex:mdplimit}}
  \label{fig:mdplimit}
  \includegraphics[scale=0.6]{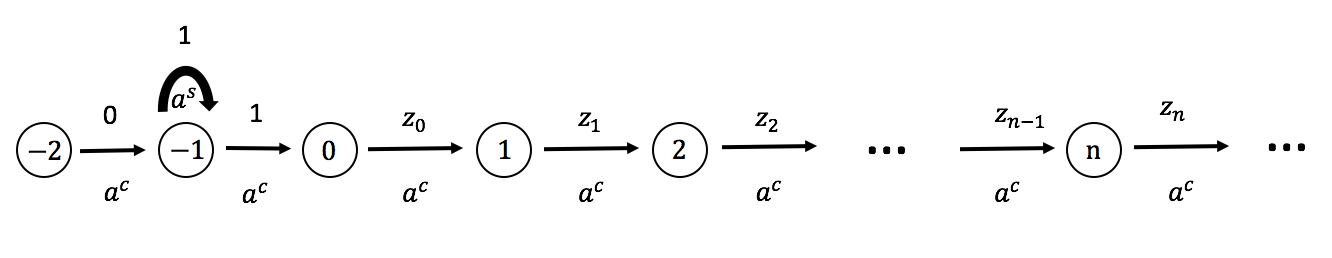}
\end{figure}

As shown in the proof of Proposition~\ref{prop:ex1:assumption} in Appendix~\ref{sec:cvg ua:apx}, the relative value
function
\begin{align}
	\ua (n) =
	\begin{cases}
		0, & \text{if } n = -2, \\
		1, & \text{if } n = -1, \\
		f(\a) + 1, & \text{if } n = 0, \\
		(1-\a)\sum_{i=0}^{\infty} z_{n+i} \a^i - z_{n-1} + 1, & \text{if } n = 1,2,\ldots\ .
	\end{cases}
	\label{eqn:ex1:ua}
\end{align}
According to Bishop~et~al.~\cite[Proposition 1]{BFZ14}, $\liminf_{\a\uparrow1}f(\a)=0$ and
$\limsup_{\a\uparrow1} f(\a) = 1.$ Hence,
$\liminf_{\a\uparrow1} \ua (0) = 1$ and $\limsup_{\a\uparrow1} \ua (0) = 2,$ that is, in this example
there exist multiple relative value functions $\tu$ defined in \eqref{EQN1}.
}
\end{example}

\begin{proposition}\label{prop:ex1:assumption}
	The MDP described in Example~\ref{ex:mdplimit} satisfies Assumptions~\textbf{W*} and \textbf{B}, where the discrete
	metric $d(x,y) = \indF_{\{x = y\}}$ is considered on $\X$ and $\A.$
\end{proposition}

Before the proof of Theorem~\ref{thm:ua converge}, we first state several properties of the functions $\bua$
defined in \eqref{eqn:def bua}. If there exists an $\a$-discounted optimal
$(\sa,\Sa)$ optimal policy, then \eqref{eqn:bva sS} implies that
\begin{align}
	\bua (x) =
	\begin{cases}
		\Ga (x) - \bma & \text{if } x \geq \sa, \\
		K & \text{if } x < \sa.
	\end{cases}
	\label{eqn:bua-2}
\end{align}

\begin{lemma}\label{lm:bua equicontinuous}
Let Assumption~\ref{assum:qcall} hold. Then,
	\begin{enumerate}[(i)]
		\item for each $\b\in (\a^*,1)$ the family of functions $\{\bua\}_{\a\in [\b, 1)}$ is equicontinuous on $\X;$ \label{lm:bua-equi}
		\item $\sup_{\a\in (\a^*,1)} \bua (x) < +\infty$ for all $x\in\X.$ \label{lm:bua-bd}
	\end{enumerate}
\end{lemma}

\begin{lemma}\label{lm:bua pointwise converge}
	Let Assumptions~\ref{assum:qcall} and \ref{assum:decrease} hold for $\a=1.$
	Then there exists the limit
	\begin{align}
	\bu (x) := \lim_{\a\uparrow1} \bu_{\a}(x) , \qquad x\in\X,
	\label{eqn:convergence-1}
\end{align}
where the function $\bu$ is continuous on $\X.$
\end{lemma}

In view of \eqref{defmauaw}, \eqref{eqn:modeltrans} and \eqref{eqn:def bua},
\begin{align}
	\ua (x) = \bua (x) + \bma - \ma - \c x, \qquad x\in\X.
	\label{eqn:ubu}
\end{align}

\begin{proof}[Proof of Theorem~\ref{thm:ua converge}]
The theorem follows from the following two statements:
\begin{enumerate}[(i)]
	\item \label{enum:thmua:1} there exists the limit $u^* (x):=\lim_{\a\uparrow1} u_{\a} (x),$ $x\in\X,$
and the function $u^*$ is continuous on $\X;$ and
	\item \label{enum:thmua:2} $u^* (x) = u (x) := \liminf_{\a\uparrow1, y\to x}u_{\a} (x)$ for all $x\in\X.$
\end{enumerate}
Let us prove statements \eqref{enum:thmua:1} and \eqref{enum:thmua:2}.
\eqref{enum:thmua:1} We first prove that there exists the limit $u^* (s^*) := \lim_{\a\uparrow1} u_{\a}(s^*),$ where $s^*$ is defined in \eqref{eqn:def s*}.

Consider $x_\a \in X_{\a},$ $\a\in[0,1),$ where $X_{\a}$ is defined in \eqref{eqn:Xa}, and
any given $\b\in (\a^*,1).$
In view of \eqref{eq:boundmaxlul}, since $X_{\a}\subset [x_L^*,x_U^*]$ for all $\a\in[0,1),$
for every sequence $\{ \a_n \uparrow 1\}_{n=1,2,\ldots},$ there exists a subsequence
$\{ \a_{n_k} \uparrow 1\}_{k=1,2,\ldots}$ of the sequence $\{ \a_n \uparrow 1\}_{n=1,2,\ldots}$ such that
$\a_{n_1}\geq \b$ and $x_{\a_{n_k}}\to x^*$ as $k\to +\infty$ for some $x^*\in[x_L^*,x_U^*].$

Consider $\eps > 0.$ Since the family of functions $\{ \ua \}_{\a\in [\b,1)}$ is equicontinuous
(see Theorem~\ref{thm:u a equicont}), there exists an integer $M(\eps) > 0$ such that for all $k \geq M(\eps)$
\begin{align}
	|u_{\a_{n_k}} (x_{\a_{n_k}}) - u_{\a_{n_k}} (x^*)| < \eps.
	\label{eqn:limit u(x*)}
\end{align}
Since $u_{\a_{n_k}} (x_{\a_{n_k}}) = 0$ for all $k=1,2,\ldots,$ \eqref{eqn:limit u(x*)} implies that
for $k \geq M(\eps)$
\begin{align}
	|u_{\a_{n_k}} (x^*) | < \eps.
	\label{eqn:uan(x*)<e}
\end{align}
Therefore, \eqref{eqn:uan(x*)<e} implies that
\begin{align}
	\lim_{k\to +\infty} u_{\a_{n_k}} (x^*) = 0 .
	\label{eqn:limit uan(x*)=0}
\end{align}
Since the function $u_{\a_{n_k}}$ is nonnegative, \eqref{eqn:uan(x*)<e} implies that for $k \geq M(\eps)$
\begin{align}
	u_{\a_{n_k}} (x^*) < u_{\a_{n_k}} (x) + \eps, \qquad x\in\X .
	\label{eqn:min value-1}
\end{align}
Then \eqref{eqn:min value-1} and \eqref{eqn:ubu} imply that for $k \geq M(\eps)$
\begin{align}
	\bar{u}_{\a_{n_k}} (x^*) - \bar{c}x^* < \bar{u}_{\a_{n_k}} (x) - \bar{c}x + \eps,
	\qquad x\in\X.
	\label{eqn:min value-2}
\end{align}
By taking the limit of both sides of \eqref{eqn:min value-2} as $k\to +\infty,$
Lemma~\ref{lm:bua pointwise converge} implies that
\begin{align}
	\bu (x^*) - \bar{c} x^* \leq \bu (x) - \bar{c}x + \eps, \qquad x\in\X .
	\label{eqn:min value-3}
\end{align}
Since $\eps$ can be chosen arbitrarily, \eqref{eqn:min value-3} implies that
\begin{align}
	\bu (x^*) - \bar{c} x^* = \min_{x\in\X} \{ \bu (x) - \bar{c}x \} .
	\label{eqn:minV-4}
\end{align}
Let $M_{\bu}:= \bu (s^*) - \bar{c}s^* - \min_{x\in\X} \{ \bu (x) - \bar{c}x \}.$
Then
\begin{align}
\begin{split}
	&\lim_{k\to +\infty} u_{\a_{n_k}} (s^*) - u_{\a_{n_k}} (x^*)  = \lim_{k\to +\infty} \bar{u}_{\a_{n_k}} (s^*) - \bar{c}s^* - [\bar{u}_{\a_{n_k}} (x^*) -  \bar{c}x^*]  \\
	=& \bu (s^*) - \bar{c}s^* -[\bu (x^*) - \bar{c}x^*]  = \bu (s^*) - \bar{c}s^* - \min_{x\in\X} \{ \bu (x) - \bar{c}x \} = M_{\bu},
\end{split}
	\label{eqn:LimitOfDiff}
\end{align}
where the first equality follows from \eqref{eqn:ubu}, the second one follows from
Lemma~\ref{lm:bua pointwise converge} and the third one follows from \eqref{eqn:minV-4}.
In view of \eqref{eqn:limit uan(x*)=0} and \eqref{eqn:LimitOfDiff},
\begin{align}
	\lim_{k\to +\infty} u_{\a_{n_k}} (s^*) = M_{\bu}.
	\label{eqn:limit u(s*)}
\end{align}
Thus, for every sequence $\{ \a_n \uparrow1 \}_{n=1,2,\ldots	}$ there exists a subsequence
$\{ \a_{n_k} \}_{k=1,2,\ldots}$ such that \eqref{eqn:limit u(s*)} holds. Therefore,
$\lim_{n\to +\infty} u_{\a_n} (s^*) = M_{\bu}$ for every sequence $\{ \a_n \uparrow1 \}_{n=1,2,\ldots	},$
which is equivalent to
\begin{align}
	u^* (s^*) := \lim_{\a\uparrow1} u_{\a} (s^*) = M_{\bu} .
	\label{eqn:limit u(s*) all}
\end{align}

Now we prove that there exists the limit $u^* (x):=\lim_{\a\uparrow1} u(x)$  for $x \in\X.$
For $x\in\X$
\begin{align}
\begin{split}
	\lim_{\a\uparrow1} \ua (x) - \ua (s^*) & =
		\lim_{\a\uparrow1} \bua (x) - \c x - [\bua (s^*) - \c s^*] \\
	& = \bu (x) - \c x - [\bu (s^*) - \c s^*],
\end{split}
	\label{eqn:limit u(y)}
\end{align}
where the first equality follows from \eqref{eqn:ubu} and the second one follows from Lemma~\ref{lm:bua pointwise converge}.
Therefore, \eqref{eqn:limit u(s*) all} and \eqref{eqn:limit u(y)} implies that
there exists the limit
\begin{align}
	u^* (x) := \lim_{\a\uparrow1} u_{\a} (x) = M_{\bu} + \bu (x) - \bar{c} x - [\bu (s^*) - \bar{c} s^*], \qquad x\in\X.
	\label{eqn:limit u(y) all}
\end{align}

Furthermore, since the family of functions $\{ \ua \}_{ \a \in [\b,1)  }$ is equicontinuous
and Assumption \textbf{B} holds, Arzel\`{a}-Ascoli theorem implies that
the function $u^*$ is continuous.

(ii) Consider sequences $\{ \a_n\uparrow1 \}_{n=1,2,\ldots}$ and $\{ y_n\to x \}_{n=1,2,\ldots}$ such that
$\a_1 > \a^* $ and
$\lim_{n\to  +\infty} u_{\a_n}(y_n) = \liminf_{\a\uparrow1, y\to x} u_{\a}(y).$ Then,
\begin{align*}
	\liminf_{\a\uparrow1, y\to x} u_{\a}(y) \leq
	\liminf_{n\to +\infty, y\to x} u_{\a_n}(y) \leq \lim_{n\to  +\infty} u_{\a_n}(y_n) = \liminf_{\a\uparrow1, y\to x} u_{\a}(x),
\end{align*}
which implies that
\begin{align}
	\liminf_{\a\uparrow1, y\to x} u_{\a}(y) = \liminf_{n\to +\infty, y\to x} u_{\a_n}(y).
	\label{eqn:lim=1}
\end{align}
According to Feinberg and Liang~\cite[Lemma~3.3]{FLi16b}, since $\lim_{n\to +\infty} u_{\a_n} (x) = u^* (x),$
Theorems~\ref{thm:B} and~\ref{thm:u a equicont} imply that
\begin{align}
	\liminf_{n\to +\infty, y\to x} u_{\a_n}(y) = \lim_{n\to +\infty} u_{\a_n} (x) = u^* (x) .
	\label{eqn:lim=2}
\end{align}
Therefore, \eqref{eqn:lim=1} and \eqref{eqn:lim=2} imply that $u:=\liminf_{\a\uparrow1, y\to x} u_{\a}(y) = u^*.$ This completes the proof.
\end{proof}

Theorem~\ref{thm:ua converge} implies that \eqref{eqn:IC ACOE-2} can be written as
\begin{align}
	H(x) := \bar{c}x + \E[h (x-D)] + \E[u (x-D)].
	\label{eqn:def H(x)}
\end{align}

\begin{corollary}\label{cor:inventory:acoe}
Let Assumptions~\ref{assum:qcall} and \ref{assum:decrease} hold for $\a=1.$
Then the conclusions of Theorem~\ref{thm:inventory:acoe} hold with $\tilde{u}=u$ defined in \eqref{eqn:ua converge-1} and $s^*$ defined in \eqref{eqn:def s*}, that is, the functions $\tilde{u}$ and the thresholds $s^*$ defined in \eqref{EQN1} and Theorem~\ref{thm:fle} respectively,
are the same for all sequences $\{ \a_n\uparrow1 \}_{n=1,2,\ldots}.$
\end{corollary}

Define the set of all possible optimal average-cost lower thresholds
\begin{align}
	\mathcal{G} := \{ x \leq S : H (y) = K + H (S) \text{ for all } y\in[s,x] \},
	\label{eqn:set of s}
\end{align}
where $S=\min\big\{ \argmin_x\{ H(x) \} \big\}$ and  $s$ is defined in \eqref{eqn:def s} with $f:=H.$
Note that $s\in \mathcal{G}$ and $y \geq s$ for all $y\in\mathcal{G}.$

The following theorem establishes the uniqueness of the optimal lower threshold satisfying the optimality equations
for the inventory  model with holding/backlog costs satisfying quasiconvexity assumptions under the average cost criterion.

\begin{theorem}\label{thm:uniqueac}
Let Assumptions~\ref{assum:qcall} and \ref{assum:decrease} hold for $\a=1.$ Then,
\begin{align*}
	\mathcal{G} = \{ s^* \},
\end{align*}
where $\mathcal{G}$ and $s^*$ are defined in \eqref{eqn:set of s} and
\eqref{eqn:def s*} respectively.
\end{theorem}

\begin{proof}
Consider $\mathcal{G}$ and $S$ defined in \eqref{eqn:set of s}.
Recall that $s \in\mathcal{G}$ and $y\geq s,$ where $s$ is defined in \eqref{eqn:def s} with $f:=H.$
According to Theorem~\ref{thm:inventory:acoe} and Corollary~\ref{cor:inventory:acoe},
for $y\in\mathcal{G}$
\begin{align}
	w + u(x) + \c x =
	\begin{cases}
		K + H(S) & \text{if } x \leq y, \\
		H(x) & \text{if } x \geq y,
	\end{cases}
	\label{eqn:Hy0}
\end{align}
which implies that for $x \leq y$
\begin{align}
\begin{split}
	H (x) &= \c x + \E[h(x - D)] + \E[u (x - D)]  \\
	&= K + \E[h(x - D)] + H (S) + \c\E[D] - w .
\end{split}
	\label{eqn:Hy1}
\end{align}
Since $H(y) = K + H(S)$ for $y\in\mathcal{G},$ in view of \eqref{eqn:Hy1},
\begin{align}
	\E[h (y - D)] = w - \c\E[D], \qquad y\in\mathcal{G} .
	\label{eqn:Hy2}
\end{align}

The following proof is by contradiction.
Assume that there exists $y_1\in \mathcal{G}$ such that $y_1 > s.$
Then \eqref{eqn:Hy2} implies that $\E[h (y_1-D)] = \E[h (s - D)]. $
Therefore, Assumption~\ref{assum:decrease} implies that $\rone < y_1,$ where $\rone$ is defined in
\eqref{eqn:def r alpha r*}. Since $S=\min\big\{ \argmin_x\{ H(x) \} \big\},$ \eqref{eqn:Hy0} implies that for $x < S$
\begin{align}
	w + u(x) + \c x > H(S).
	\label{eqn:ry1}
\end{align}
Therefore,
\begin{align}
\begin{split}
	H (y_1) &= K + H(S) = K + \c S + \E[h(S-D)] + \E[u(x-D)]  \\
	&> K + \E[h(\rone - D)] + H(S) + \c\E[D] -w,
\end{split}
	\label{eqn:rycontra}
\end{align}
where the first equality holds because $y_1\in\mathcal{G},$ the second follows from \eqref{eqn:def H(x)}, and the inequality holds because $ \E[h(S-D)]\geq  \E[h(\rone-D)],$
\eqref{eqn:ry1} and $\PR (D>0) > 0.$ Since $y\in\mathcal{G}$ and $\rone < y_1,$
$H(\rone) \geq H(y_1).$ In view of \eqref{eqn:Hy1},
\begin{align*}
	H(y_1) \leq H(\rone)=  K + \E[h (\rone - D) ] + H(S) + \c\E[D] - w ,
\end{align*}
which contradicts \eqref{eqn:rycontra}. Then, $\mathcal{G} = \{ s \}.$
In addition, Corollary~\ref{cor:inventory:acoe} implies that $s^*\in \mathcal{G},$
where $s^*$ is defined in \eqref{eqn:def s*}.
Therefore, $s=s^*$ and $\mathcal{G} = \{ s^* \}.$
\end{proof}

The following corollary states that all the results of this paper
hold for inventory  models with convex holding/backlog costs.

\begin{corollary}\label{cor:inventory convex}
	The conclusions of lemmas, theorems, and corollaries in Sections~\ref{sec:cvg sa} and \ref{sec:cvg ua} hold under Assumption~\ref{assum:convex}.
\end{corollary}

\section{Veinott's Reduction of Problems with Backorders and Positive Lead Times to Problems without Lead Times}
\label{sec:veinott}

In this section, we explain, by using the technique introduced by Veinott~\cite{Vei66} for finite-horizon problems
with continuous demand and without formal proofs, that the infinite-horizon inventory model with positive lead times and backorders 
can be reduced to the model without lead times. Therefore, the results of this paper, Feinberg and Lewis~\cite{FL16},
and Feinberg and Liang~\cite{FLi16a,FLi16b} also hold for the inventory  model with positive lead times. For inventory
model with a positive lead times, we also provide a formal formulation of the MDP with transformed state space for
future reference.

Consider the inventory  model defined in Section~\ref{sec:inventory control}. Instead of
assuming zero lead times, assume that the fixed lead time is $L\in \N := \{1,2,\ldots\},$ that is,
an order placed at the beginning of time $t$ will be delivered at the beginning of time $t+L.$
In addition, let $h^L(x)$ be the holding/backlog cost per period if the inventory level is $x.$
We define
\begin{align}
	h^* (x) : = \E[h^L (x-\sum_{i=1}^L D_i)] .
	\label{eqn:h*}
\end{align}

For the inventory  model, the dynamics of the system are defined by the equation
\begin{align}
	x_{t+1} = x_t + a_{t-L} - D_{t+1}, \qquad t=0,1,2,\ldots ,
	\label{eqn:dynamic1}
\end{align}
where $x_t$ and $a_t$ are the current inventory level before replenishment and the ordered
amount at period $t.$ In addition, at period $t,$ the one-step cost is
\begin{align}
	\tilde{c} (\bh^L_t,a_t) := K \indF_{\{ a_{t-L} > 0 \}} + \c a_{t-L} + \E[h^L (x_t + a_{t-L} - D_{t+1})] ,
	\qquad t = 0,1,\ldots,
	\label{cost1}
\end{align}
where $ \bh^L_t = (a_{-L},a_{-L+1},\ldots,a_{-1},x_0,a_0,\ldots,a_{t-1},x_t)$ is the history at
period $t.$

Equation~\eqref{eqn:dynamic1} means that a decision-maker observes at the end of the
period $t$ the history $h_t,$ places an order of amount $a_t,$ which will be delivered
in $L$ periods (that is, at the end of the period $t+L$), and the demand occurred during the
period $t+1$ is $D_{t+1}.$

As usual, consider the set of possible trajectories
$\bh^L_{ +\infty} = (\bh^L_t,a_t,x_{t+1},a_{t+1},\ldots) .$ An arbitrary policy
is a regular probability distribution $\pi (d a_{t}|\bh^L_t),$ $t=0,1,\ldots,$ on
$\Rp.$ It defines the transition probability for $\bh^L_t$ to $(\bh^L_t,a_{t}).$
The transition probability for $(\bh^L_t,a_{t})$ can be defined by \eqref{eqn:dynamic1}.
Therefore, given the initial state $\bh^L_0 = (a_{-L},a_{-L+1},\ldots,a_{-1},x_0),$
a policy $\pi$ defines, in view of the Ionescu Tulcea theorem, the probability distribution
$\PR_{\bh^L_0}^{\pi}$ on the set of trajectories. We denote by $\E_{\bh^L_0}^{\pi}$
the expectation with respect to $\PR_{\bh^L_0}^{\pi}.$

For a finite-horizon $N=1,2,\dots$ the expected total discounted cost is
\begin{align}\label{eqn:costs2y}
    \tilde{v}_{N,\a}^{\pi} (\bh^L_0) := \E_{\bh^L_0}^{\pi} \Big[ \sum_{t=0}^{N-1}
    \alpha^{t} \tilde{c}(x_t,a_t) \Big]
    = \E_{\bh^L_0}^{\pi} \Big[  \sum_{t=0}^{L-1} \a^t \tilde{c}(x_t,a_t) + \a^L \sum_{t=0}^{N-1} \a^t \tilde{c}(x_{t+L},a_{t+L}) \Big] ,
\end{align}
where $\alpha\in [0,1]$ is the discount factor and $\tilde{v}_{0,\a}^{\pi} (\bh^L_0)= 0.$
When $N= +\infty$ and $\a\in [0,1),$ \eqref{eqn:sec_model def:finite total disc cost} defines the
infinite-horizon expected total discounted cost denoted by $\tilde{v}_{\a}^{\pi}(\bh^L_0).$
The \emph{average cost per unit time} is defined as
$\tilde{w}^{\pi} (h_0):=\limsup_{N\to +\infty} \frac{1}{N}\tilde{v}_{N,1}^{\pi} (\bh^L_0).$

Let us define \begin{align}
	y_t := x_t + \sum_{i=1}^{L}a_{t-i} = x_{t+L} + \sum_{i=1}^L D_{t+i}, \qquad t=0,1,\ldots,
	\label{xy}
\end{align}
where $y_t$ is the sum of the current inventory level and the outstanding orders at the end
of period $t.$ Since the distribution of $x_{t+L}$ is determined by $y_t,$
in view of \eqref{eqn:dynamic1}, we show that it is possible to make the decision $a_t$ only based on the
quantity $y_t.$

Let us construct an MDP with state space $\Y = \R$ (or $\Y=\Z$) with states $y_t$ defined in \eqref{xy}.
The actions are the amount of orders that can be placed at each period $t;$ $A(y) = \A = \Rp$ (or $A(y) = \A=\N_0$)
for all $y\in\Y.$ In view of \eqref{eqn:dynamic1}
and \eqref{xy}, the dynamics of the system are defined by the equation
\begin{align}
	y_{t+1} = y_t + a_{t} - D_{t+1}, \quad t=0,1,2,\dots\ .
	\label{eqn:dynamic2y}
\end{align}
The transition probabilities for the MDP corresponding to \eqref{eqn:dynamic2y} is
\begin{align}
	q^*(B | y_t, a_t) = \PR (y_t + a_t - D_{t+1} \in B) ,
	\label{eqn:tranp2y}
\end{align}
for each measurable subset $B$ of $\Y .$ Let the one-step cost be
\begin{align}
	c^*(y,a) := K\indF_{\{ a > 0 \}} + & \c a + \E[h^*(y + a - D)] .
	\label{eqn:cy}
\end{align}

As was noticed by Veinott~\cite{Vei66}, the sum of current and ordered inventories forms an MDP whose dynamics 
and costs are defined by expressions \eqref{eqn:dynamic2y}--\eqref{eqn:cy}. So, we have the same MDP as for
in the problem without lead time with the only difference that the holding/backlog cost function $h$ is substituted
with the function $h^*$.
In addition, though the amount of inventory $x_{t+L}$ at time $(t+L)$ is not known at
time $t,$ the distribution of $x_{t+L}$ is known because $x_{t+L} \sim y_t - \sum_{l=1}^L D^{(l)},$ where
$D^{(1)},\ldots,D^{(L)}$ are i.i.d. random variables with $D^{(l)} \sim D,$ $l=1,2,\ldots,L.$ 
Since the actual amount of inventory level $x_{t+L}$  is unknown at time $t$, when the amount $a_t$  is ordered, 
 this problem can be modeled as a Partially Observable Markov Decision Process (POMDP).
According to current available theory (see Hern\'{a}ndez-Lerma~\cite[Chapter 4]{Her89},
Feinberg et al.~\cite{FKZ16}, and references therein), such models can be reduced to the MDPs whose states are
probability distributions of $x_{t+L}$ known at time $t,$ which is the distribution of $y_t - \sum_{l=1}^L D^{(l)} .$
Therefore, optimal policies for the MDP introduced by Veinott~\cite{Vei66} with state space $\Y,$ action space 
$\A,$ and transition probabilities \eqref{eqn:tranp2y}, and costs \eqref{eqn:cy}  
define the optimal
actions at time $t=0,1,2,\ldots\ .$ 

\begin{theorem}\label{thm:veinott1}
	Consider the problem with the lead time $L=1,2,\ldots\ .$ Then the MDP $\{ \Y,\A,q^*,c^* \}$ coincides
	the MDP $\{ \X,\A,q,c \}$ with the function $h$ substituted with $h^*$ in the latter one. Therefore, the
	conclusions of theorems in this paper
	hold for the problems with the lead time $L=1,2,\ldots,$ if the holding/backlog cost function $h^*$
	satisfies the conditions assumed for the function $h$ in the corresponding statements. 
\end{theorem}

\begin{proof}
	Since $y_t\in\X$ and the actions are the same for these two models, we need to verify only
	the correspondence for transition probabilities and costs. If $h = h^*,$ then formulae
	\eqref{eqn:c} and \eqref{eqn:cy} coincide with $x_t = y_t.$ The transition probabilities $q^*$ defined in
	\eqref{eqn:tranp2y} also coincides with \eqref{eqn:tranp2}.
	Observe that it is easy to show that
\begin{align}
	\tilde{v}_{N,\a}^{\pi} (\bh^L_0) = f(\bh^L_0) + \a^L v_{N,\a}^{\pi} (y_0),
	\label{eqn:costs2}
\end{align}	
where $f(\bh^L_0) := \sum_{t=0}^{L-1} \a^t \E[\tilde{c} (x_0 + \sum_{i=0}^{t-1} a_{-L+i} - \sum_{i=0}^{t-1} D_{t+i}, a_t ) ] .$
\end{proof}

For the problems with convex holding/backlog cost function $h^L,$ the function $h^*$ is also convex and
$\E[h^* (x-D)] \to +\infty$ as $|x|\to +\infty.$
We also need the addition assumption that $\E[h^* (x-D)] < +\infty$ for all
$x\in\X.$ Then the results in this paper formulated under Assumption~\ref{assum:convex} and the results in
Feinberg and Lewis~\cite{FL16} and Feinberg and Liang~\cite{FLi16a, FLi16b} hold for the problems with
the lead time $L=1,2,\ldots\ .$

\begin{remark}\label{rm:finiteh}
{\rm
Note that the assumption on the finiteness of the function $\E[h^* (x-D)]$ is necessary for the problems with
convex holding/backlog costs. Consider the lead time $L=1,$ the holding/backlog cost function
\begin{align*}
	h^L (x) :=
	\begin{cases}
		x + \frac{e^{2}}{5}, & \text{if } x > 0, \\
		\frac{e^{-x+2}}{(x-2)^2 + 1}, & \text{if } x \leq 0 ,
	\end{cases}
\end{align*}
and the random variable $D$ follows the exponential distribution with density function
$f_D (x) = e^{-x},$ if $x > 0,$ and $f_D (x) = 0,$ otherwise. Then the random variable $\textbf{S}_2$
follows the Erlang distribution with density function $f_{\textbf{S}_2} (x) = xe^{-x},$ if $x > 0,$ and
$f_{\textbf{S}_2} (x)= 0,$ otherwise. Observe that the function $h^L$ is continuous and nonnegative.
Some calculations show that the function $h^L$ is convex on $\R$ and $\E[h^L (x-D)] < +\infty$
for all $x\in\R.$ However, $\E[h^* (0-D)] = \E[h^L (0-\textbf{S}_2)] = +\infty.$
}
\end{remark}

\begin{remark}\label{rm:lostsales}
{\rm
	The reduction discussed in this section does not hold for the inventory
	model with lost-sales. For such model with lead time $L > 0,$ the dynamics of the system are defined
	by the equation
	\begin{align*}
		x_{t+1} = (x_t + a_{t-L} - D_{t+1})^+ :=\max\{x_t + a_{t-L} - D_{t+1} ,0\},\qquad t = 0,1,2,\ldots\ .
	\end{align*}
	Consider the transformation similar to the one defined in \eqref{xy}. Then
	$x_{t+L} = y_t - \sum_{i=1}^L \tilde{D}_{t+i},$ where $\tilde{D}_{j} :=\min\{ D_j,x_{j-1}+a_{j-L-1} \},$
	$j=t+1,t+2,\ldots,t+L.$ Since the distribution of $x_{t+L}$ does not depend solely on the information
	available at time $t,$ the reduction does not hold. Indeed, the structure of the optimal policies
	may depend on the lead times. In particular, if the lead times are large,
	then the constant-order policy performs nearly optimally; see Goldberg~et~al.~\cite{GKLSS}.
}
\end{remark}

\noindent
{\bf Acknowledgement.}
This research  was partially supported by NSF grants CMMI-1335296 and CMMI-1636193.

\newpage
\begin{appendices}
\section{Proofs to Section~\ref{sec:inventory control}}
\label{sec:inventory control:apx}

\begin{proof}[Proof of Lemma~\ref{lm:cqc}]
Since $\E[h(x-D)]\to +\infty$ as $x\to -\infty,$
\begin{align}
\limsup_{x\to -\infty }h(x) = +\infty.
\label{eqn:hinfty}
\end{align}
To see this note if $\limsup_{x\to -\infty} h(x) < +\infty,$ then there exist real numbers $M_1,$ $M_2 > 0$ such that
$ h (x)  \leq M_1$ for $x \leq -M_2.$ Since $D$ is a nonnegative random variable, $\E[h(x-D)] \leq M_1$ for
$x\leq -M_2$ and $\limsup_{x\to -\infty} \E[h(x-D)] \leq M_1 < +\infty.$ This contradicts the assumption that
$\E[h(x-D)]\to +\infty$ as $x \to -\infty.$

Since the function $h$ is convex,  the function $\E[h(x-D)]$ is convex.
Therefore, in view of \eqref{eqn:def ha}, the function $\E[\h (x-D)]$ is convex for all $\a\in[0,1].$
Since every convex function is quasiconvex,  the function $\E[\h (x-D)]$ is quasiconvex for all $\a\in[0,1].$

Since the function $h$ is convex on $\X,$ it is continuous. Therefore, \eqref{eqn:hinfty} implies
\begin{align}
	\lim_{x\to -\infty} h(x) = +\infty .
	\label{eqn:hinfty1}
\end{align}
As explained in Feinberg and Liang~\cite[Equations (2.3), (4.1)]{FLi16a},
\eqref{eqn:hinfty1} and the convexity of the function $h$ imply that
$1+\lim_{x\to -\infty} \frac{h(x)}{\c x} < 1.$

Consider $\a^* \in [\max\{ 1 + \lim_{x\to -\infty} \frac{h(x)}{\c x},0 \},1).$
For $\a\in(\a^*,1],$ since the function $\h (x) = \c x(\frac{h(x)}{\c x}+1-\a)$ tends to $+\infty$ as $x\to-\infty,$
\begin{align}
	\lim_{x\to -\infty}\E[h_{\a}(x-D)] =  +\infty, \qquad \a\in (\a^*,1].
	\label{eqn:limitha:1}
\end{align}
Therefore, the convexity of the function $\E[h_{\a} (x-D)]$ implies that
$\lim_{x\to -\infty}\E[\h (x-D)] > K + \inf_{x\in\X} \E[\h (x-D)] $ for all $\a\in(\a^*,1].$
Hence, Assumption~\ref{assum:qcall} holds with $\a^* \in [\max\{1 + \lim_{x\to -\infty} \frac{h(x)}{\c x},0\},1).$
In view of \eqref{eqn:limitha:1}, the convexity of the function
$\E[h_{\a} (x-D)]$ implies that Assumption~\ref{assum:decrease} holds for all
$\a\in(\a^*,1].$
\end{proof}

\begin{proof}[Proof of Lemma~\ref{lm:assumption}]
It is straightforward that Assumption~\ref{assum:qcall} implies Assumption~\ref{assum:qcalpha} for $\a\in (\a^*,1).$
In addition, since $\E[\h (x-D)] \to \E[h (x-D)]$ as $\a\uparrow1$ for all $x\in\X,$ the quasiconvexity of the function
$\E[\h (x-D)]$ implies that the function $\E[h(x-D)]$ is quasiconvex. Since $\inf_{x\in\X} \E[h (x-D)] < +\infty $ and
$\E[h (x-D)] = \E[\h (x-D)] - (1-\a)\c x\to +\infty$ as $x\to-\infty$ for each $\a\in (\a^*,1),$
Assumption~\ref{assum:qcalpha} holds for $\a = 1.$
\end{proof}

\section{Proofs to Section~\ref{sec:mdptotal}}
\label{sec:mdptotal:apx}

\begin{proof}[Proof of Lemma~\ref{lm:sS1}]
In view of \eqref{eqn:vna:trans}, for $x\leq y$ and $t=1,2,\ldots$
\begin{align*}
	\bv_{t,\a} (x) &=\min\{ \min_{a\geq 0}\{ K + \bG_{t-1,\a}(x+a) \},\bG_{t-1,\a} (x) \}\leq  \min_{a\geq 0} \{ K+ \bG_{t-1,\a} (x + a) \} \\
	& \leq K + \min_{a\geq y-x} \bG_{t-1,\a} (x + a) = K + \min_{a\geq 0} \bG_{t-1,\a} (y + a) \\
	& \leq K + \min\{ \min_{a\geq 0}\{ K + \bG_{t-1,\a}(y+a) \},\bG_{t-1,\a} (y) \} = K + \bv_{t,\a} (y)  ,
\end{align*}
where the second inequality follows from $y-x \geq 0.$
Furthermore, \eqref{eqn:va:trans} and the same arguments imply \eqref{eqn:va1}.

In view of \eqref{eqn:gna:trans}, for $x\leq y$ and $t=1,2,\ldots $
\begin{align*}
	\bG_{t,\a} (y) - \bG_{t,\a} (x) =& \E[\h(y-D)] - \E[\h(x-D)]  +\a\E[\bv_{t,\a} (y-D)  - \bv_{t,\a} (x-D)  ] \\
	\leq & \E[\h(y-D)] - \E[\h(x-D)] + \a K,
\end{align*}
where the inequality follows from \eqref{eqn:vta1}. Furthermore, \eqref{eqn:ga:trans} and the same arguments
imply \eqref{eqn:Ga1}.
\end{proof}

\begin{proof}[Proof of Lemma~\ref{lm:order}]
The proof is by induction on $t.$ For $t = 0,$ \eqref{eqn:vtade1} holds because $\bv_{0,\a} (x) = 0,$ $x\in\X,$ and
\eqref{eqn:Gtade1} follows from $\bG_{0,\a} (x) = \E[\h (x-D)],$ $x\in\X,$ and Assumption~\ref{assum:qcalpha}.
To complete the induction arguments, assume that \eqref{eqn:Gtade1} holds for $t = k \in \{0,1,2,\ldots\}.$
Then for $x\leq y\leq\ra$
\begin{align*}
	& \bv_{k+1,\a} (x) = \min\big\{ \bG_{k,\a} (x), \min_{a\geq 0}\{ K+\bG_{k,\a} (x+a)\} \big\} \\
	 = & \min\big\{ \bG_{k,\a} (x), \min_{a\geq 0}\{ K+\bG_{k,\a} (x+a)\} \big\} \\
	\geq & \min\big\{ \bG_{k,\a} (y), \min_{0\leq a< y-x}\{ K+\bG_{k,\a} (x+a)\}, \min_{a\geq y-x}\{ K+\bG_{k,\a} (x+a) \}  \big\} \\
	\geq & \min\big\{ \bG_{k,\a} (y), K+\bG_{k,\a}(y) , \min_{a\geq 0}\{ K+\bG_{k,\a} (y+a)\} \big\} \\	
	\geq & \min\big\{ \bG_{k,\a} (y) ,\min_{a\geq0}\{ K+\bG_{k,\a}(y+a)\}  \big\}   = \bv_{k+1,\a} (y)  ,
\end{align*}
where the first and last equalities follow from \eqref{eqn:vna:trans}, the first two inequalities  follow from \eqref{eqn:Gtade1}, and the last inequality follows from $K>0.$
Thus, \eqref{eqn:vtade1} holds for $t=k+1.$ In addition, for $x\leq y \leq \ra$
\begin{align*}
	&\bG_{k+1,\a}(y) - \bG_{k+1,\a} (x) \\
	= & \E[\h (y-D)] - \E[\h (x-D)] +\a\E[\bv_{k+1,\a}(y-D) - \bv_{k+1,\a} (x-D)  ] \leq   0,
\end{align*}
where the equality follows from \eqref{eqn:gna:trans} and the inequality holds because the function $\E[\h (x-D)]$
is nonincreasing on $(-\infty,\ra]$ and \eqref{eqn:vtade1}.
\end{proof}

\begin{proof}[Proof of Lemma~\ref{lm:vaorder}]
Note that the equality in the Lemma~\ref{lm:vaorder} follows from \eqref{eqn:trans}.
For the model with zero unit and terminal costs, since $\bv_{t,\a} \to \bva$ as $t\to+\infty,$
\eqref{eqn:vtade1} implies \eqref{eqn:vade1}. In addition, for $x\leq y \leq \ra$
\begin{align*}
	\bGa (y) -\bGa (x) = \E[\h (y-D)] - \E[\h (x-D)] +\a\E[\bva (y-D)- \bva (x-D)  ] \leq   0,
\end{align*}
where the equality follows from \eqref{eqn:ga:trans} and the inequality holds because the function $\E[\h (x-D)]$
is nonincreasing on $(-\infty,\ra]$ and \eqref{eqn:vade1}.
\end{proof}

\begin{proof}[Proof of Lemma~\ref{lm:samevalues}]
We first prove that
	\begin{align}
		\bva (x) - \c x = \tva (x) \geq \va (x) \geq  0, \qquad x\in\X.
		\label{eqn:tva>=0}
	\end{align}
Note that $\va (x) \geq 0$ for all $x\in\X,$ because all costs in the original inventory model are nonnegative.
Theorem~\ref{thm:qcall:t}\eqref{thm:qcall:t:1}, \eqref{eqn:trans:f}, and \eqref{eqn:bga=tga} imply that
for $N = 1,2,\ldots$
\begin{align}
	\tv_{N,\a} (x) = \tv_{N,\a}^{\phi^N} (x), \qquad x\in\X,
	\label{eqn:tvtaval}
\end{align}
where the policy $\phi^N$ is the $(s_{t,\a},S_{t,\a})_{t = 0,1,2,\ldots,N-1}$ policy defined in Theorem~\ref{thm:qcall:t}\eqref{thm:qcall:t:1}. Therefore,
\begin{align}
\begin{split}
	\tv_{N,\a} (x) &= \tv_{N,\a}^{\phi^N} (x) = \E_x^{\phi^N} \big[\sum_{t=0}^{N-1} c(x_t,a_t) - \a^N \c x_N \big] \\
	&= v_{N,\a}^{\phi^N} (x)  -  \a^N \c \E_x^{\phi^N}[x_N] \geq  v_{N,\a} (x)-  \a^N \c \max\{x,\barS_\a\}, \qquad x\in\X,
\end{split}
	\label{eqn:tvasS}
\end{align}
where the last inequality holds because $v_{N,\a}^{\phi^N} (x)  \geq v_{N,\a} (x),$ $x\in\X,$ and for all $N = 1,2,\ldots$ Theorem~\ref{thm:qcall:t}\eqref{thm:qcall:t:3} implies
that $\E_x^{\phi^N} [x_N] \leq \max\{ x, \barS_\a \} .$  Hence,
\begin{align*}
	\bva (x) - \c x=\tva (x) =\lim_{N\to+\infty} \tv_{N,\a} (x) \geq \lim_{N\to+\infty} v_{N,\a} (x) = \va (x) \geq 0 ,
\end{align*}
where the first two equalities follow from \eqref{eqn:trans}, the first inequality follows from \eqref{eqn:tvasS} and
$\lim_{N\to+\infty} \a^N \c \max\{x,\barS_\a\} = 0$ for each $x\in\X.$
Therefore, \eqref{eqn:tva>=0} holds.

To prove Lemma~\ref{lm:samevalues}, it remains to prove that $\va (x) \geq \tva (x),$ $x\in\X. $
Observe that for $t=0,1,2,\ldots$ and $\pi\in\PS$
\begin{align*}
	x_t \geq x_0 - \textbf{S}_t \quad \text{ and } \quad \E_{x_0}^{\pi}[x_t] \geq x_0 - t\E[D] ,
\end{align*}
where $\textbf{S}_t$ is defined in \eqref{eqn:Sn}.
Then for $N = 1,2,\ldots$
\begin{align}
\begin{split}
	\tv_{N,\a} (x) &\leq \tv^{\pi}_{N,\a} (x) = \E_x^{\pi} \big[\sum_{t=0}^{N-1} c(x_t,a_t) - \a^N \c x_N \big] = v_{N,\a}^{\pi} (x)  -  \a^N \c \E_x^{\pi}[x_N]  \\
	& \leq  v^{\pi}_{N,\a} (x)-  \a^N \c (x - N\E[D]), \qquad x\in\X .
\end{split}
	\label{eqn:tvasS:1}
\end{align}
Observe that $\lim_{N\to+\infty} \a^N \c (x - N\E[D]) = 0$ for each $x\in\X.$ Thus,
by taking the limits as $N\to+\infty$ of both sides of \eqref{eqn:tvasS:1},
$\tv_{\a} (x)  \leq v^{\pi}_{\a} (x)$ for all $\pi\in\PS,$ which implies that
$\tva (x) \leq \va (x),$ $x\in\X.$ Hence, $\tva (x) = \va (x) = \bva (x) - \c x,$ $x\in\X.$
\end{proof}

\section{Proofs to Section~\ref{sec:acoe}}
\label{sec:acoe:apx}

\begin{proof}[Proof of Lemma~\ref{lm:Horder}]
(i) In view of \eqref{eqn:va1} and Lemma~\ref{lm:samevalues}, \eqref{eqn:tu1} holds because $\ua (x) - \ua (y) = \va (x) - \va (y) = \bva (x) -\c x - (\bva (y) - \c y)$ and the function
$u_{\a_n} (x)$ converges pointwise to $\tu (x)$ as $n\to +\infty.$ For $x\leq y,$
\begin{align*}
	H (y) - H (x) =& \E[h(y-D)] +\a\E[\tu (y-D) + \c (y-D)]  \\
	& - \E[h(x-D)]-\a\E[\tu (x-D) + \c (x-D)] \\
	\leq & \E[\h(y-D)] - \E[\h(x-D)] + \a K,
\end{align*}
where the equality follows from \eqref{eqn:IC ACOE-2} and the inequality follows from \eqref{eqn:tu1}.

(ii) We first show that for $1\geq \a \geq \b > \a^*$
\begin{align}
	\ra \geq \rb .
	\label{eqn:ra<rb}
\end{align}
To verify this inequality, consider $1\geq \a\geq \b>\a^*.$ Then, for $x < \rb$
\begin{align}
\begin{split}
  \E & [\h (x-D)] - \E[\h(\rb - D)]  \\
= \E & [h(x-D)] + \E[h(\rb - D)] + (1-\a)\c(x-\rb)  \\
> \E & [h(x-D)] + \E[h(\rb - D)] + (1-\b)\c(x-\rb)  \\
= \E & [h_\b (x-D)] - \E[h_\b(\rb - D)] > 0,
\end{split}
\label{eqn:raincrease}
\end{align}
where the equalities follow from \eqref{eqn:def ha} and the inequality holds because $1-\a < 1-\b$ and $\c(x-\rb) < 0.$ If $\ra < \rb,$ then \eqref{eqn:raincrease} with $x = \ra$ implies that
\begin{align}
	\E[\h (\ra -D)] - \E[\h (\rb -D)] > 0 .
	\label{eqn:raincrease:contra}
\end{align}
However, the definition of $\ra$ in \eqref{eqn:def r alpha r*} implies that $\E[\h (\ra -D)] - \E[\h (\rb -D)] \leq 0 ,$ which contradicts \eqref{eqn:raincrease:contra}. Therefore, \eqref{eqn:ra<rb} holds.

Now, we prove that $\ra\uparrow\rone$ as $\a\uparrow 1.$ Consider a fixed discount factor $\b\in (\a^*,1).$
In view of \eqref{eqn:ra<rb}, since $\rb\leq \ra \leq \rone$ for all $\a\in (\b,1)$ and $|\rb|,$ $|\rone|<+\infty,$
the monotone convergence theorem implies that there exists $\ronea\in[\rb,\rone]$ such that $\ra\uparrow\ronea$ as
$\a\uparrow1. $ In addition, for $\a\in (\b,1)$
\begin{align}
\begin{split}
   0 & \leq  \E [h_1 (\ra-D)] - \E[h_1 (\rone - D)] \\
   &= \E [\h (\ra-D)] -  (1-\a)\c (\ra - \E[D])- \E[h_1 (\rone - D)] \\
   &\leq \E [\h (\rone-D)] -  (1-\a)\c (\ra - \E[D]) - \E[h_1(\rone - D)]  \\
&=  (1-\a)\c (\rone - \ra) \leq (1-\a)\c (\rone - \rb),
\end{split}
\label{eqn:raincrease:limit}
\end{align}
where the first two inequalities follows from the definition of $\rone$ and $\ra$ in \eqref{eqn:def r alpha r*},
the first equality holds because $\E [h_1 (\ra-D)] = \E [\h (\ra-D)] -  (1-\a)\c (\ra - \E[D]),$
the second equality holds because $\E [\h (\rone-D)] -\E [h_1 (\rone-D)] =  (1-\a)\c (\rone - \E[D]),$ and the last
inequality follows from  $\rb\leq \ra$ and $(1-\a)\c>0$ for $\a\in (\b,1).$ Observe that $(1-\a)\c (\rone - \rb)\to 0$ as $\a\uparrow1.$ Then
\eqref{eqn:raincrease:limit} implies that $\lim_{\a\uparrow1} \E [h_1 (\ra-D)] - \E[h_1 (\rone - D)] = 0.$
Therefore, the continuity of the function $\E[h_1 (x-D)]$ implies that
\begin{align}
	\E [h_1 (\ronea-D)] = \E[h_1 (\rone - D)].
	\label{eqn:raincrease:limit:contra}
\end{align}
Recall that $\ronea\in[\rb,\rone].$ If $\ronea < \rone,$ then Assumption~\ref{assum:qcall} and \eqref{eqn:def r alpha r*} imply that
$\E [h_1 (\ronea-D)] > \E[h_1 (\rone - D)],$ which contradicts \eqref{eqn:raincrease:limit:contra}. Therefore, $\ronea = \rone$ and $\ra\uparrow\rone$ as $\a\uparrow 1.$

In view of \eqref{eqn:vade1} and Lemma~\ref{lm:samevalues}, \eqref{eqn:tude1} holds because $\ua (x) - \ua (y) = \va (x) - \va (y)= \bva (x) -\c x - (\bva (y) - \c y),$ the function $u_{\a_n} (x)$ converges pointwise to $\tu (x)$ as $n\to +\infty,$
and $\ra\uparrow\rone$ as $\a\uparrow 1.$ For $x\leq y \leq \rone,$
\begin{align*}
 H(y) - H (x) = & \E[h (y-D)] - \E[h (x-D)] \\
  & + \a\E[\tu(y-D) + \c (y-D) - \tu (x-D) - \c (x-D) ] \leq  0,
\end{align*}
where the equality follows from \eqref{eqn:IC ACOE-2} and the inequality follows from
that the function $\E[h (x-D)]$ is nonincreasing on $(-\infty,\rone]$ and \eqref{eqn:tude1}.
\end{proof}

\begin{proof}[Proof of Corollary~\ref{cor:ordering at sa}]
The proof of the optimality of $(s,S)$ policies is based on the fact that
$K+H(S)<H(x),$ if $x<s,$ and $K+H(S)\ge H(x),$ if $x\ge s.$  Since the function $H$ is continuous,
we have that $K+H(S)=H(s).$ Thus both actions are optimal at the state $s.$
\end{proof}

\section{Proofs to Section~\ref{sec:cvg sa}}
\label{sec:cvg sa:apx}

\begin{proof}[Proof of Lemma~\ref{lm:sa}]
According to \eqref{eqn:bva sS}, $\bva (x) = K + \Ga (\Sa) = K + \bma$ for $x \leq y.$
In view of \eqref{eqn:ga:trans}, \eqref{eqn:set of sa}, and Lemma~\ref{lm:samevalues},
\begin{align*}
	K + \bma = \Ga (y) = \E[\h (y - D)] + \a \E[\bva (y - D)] = \E[\h (y - D)] + \a (K+\bma),
\end{align*}
which implies \eqref{eqn:value and ha}.
\end{proof}

\begin{proof}[Proof of Lemma~\ref{lm:r alpha <= r*}]
Observe that the second inequality in Lemma~\ref{lm:r alpha <= r*} follows from \eqref{eqn:ra<rb}.
The following proof is by contradiction. 	
Assume that there exist $\a\in (\a^*,1)$ and $y \in\mathcal{G}_{\a}$ such that $y > \ra.$
According to \eqref{eqn:bva sS}, $\bva (x) = K + \bma$ for $x \leq y.$
Therefore, \eqref{eqn:ga:trans} and Lemma~\ref{lm:samevalues} imply that for $x\leq y$
\begin{align}
	\Ga (x) = \E[\h (x - D)] + \a (K + \bma) .
	\label{eqn:Ga(ra)}
\end{align}
The definition \eqref{eqn:def r alpha r*} of $\ra$ and \eqref{eqn:Ga(ra)} imply that $\Ga(\ra)\leq \Ga(y).$
According to the definition of $\sa$ in \eqref{eqn:def s},
$\Ga(x) \geq \Ga(y)$ for $x\leq y.$ Therefore, $\Ga(\ra) = \Ga(y),$ which implies that
\begin{align}
\begin{split}
	& \E[\h (\ra - D)] + \a (K + \bma) = \Ga(\ra) = \Ga (y) = K + \Ga(\Sa)  \\
	= & K + \E[\h (\Sa - D)] + \a \E[\bva (\Sa - D)] > \E[\h(\ra - D)] + \a (K + \bma),
\end{split}
	\label{eqn:contra y>ra}
\end{align}
where the first equality follows from \eqref{eqn:Ga(ra)}, the last equality follows from \eqref{eqn:ga:trans} and Lemma~\ref{lm:samevalues},
and the inequality follows from $K > \a K$ and the definition of $\ra$ and $\bma.$
The contradiction in \eqref{eqn:contra y>ra} implies that $y \leq \ra$
for all $y \in \mathcal{G}_\a.$
\end{proof}

\begin{proof}[Proof of Lemma~\ref{lm:limit m*}]
According to equation~\eqref{eqn:sab}, for any given $\b\in (\a^*,1),$
there exists a constant $b>0$ such that $s_{\a}\in (-b,b)$ for all $\a\in [\b,1).$
In view of \eqref{eqn:bar w = underline w}, since $b$ and $x_U^*$ are real numbers,
where  $x^*_U$ is defined in \eqref{eq:boundmaxlul},
\begin{align}
\lim_{\a\uparrow1}(1-\a)(\ma - \c b) = \lim_{\a\uparrow1} (1-\a) (\ma + \c x_U^*) = w.
\label{eqn:limit bounds of bma}
\end{align}
Therefore, since $\sa > -b$ for all $\a\in[\b,1),$
\eqref{eqn:bounds on bma} and \eqref{eqn:limit bounds of bma} imply that
\begin{align}
	\lim_{\a\uparrow1} (1-\a) \bma =  w .
	\label{eqn:limit (1-a) bma}
\end{align}
Therefore, \eqref{eqn:limit (1-a) bma} and
Lemma~\ref{lm:sa} implies that
$\lim_{\a\uparrow1}\E[h_{\a} (s_{\a} - D)] = w .$
\end{proof}

\section{Proofs to Section~\ref{sec:cvg ua}}
\label{sec:cvg ua:apx}

\begin{proof}[Proof of Proposition~\ref{prop:ex1:assumption}]
We first verify the validity of Assumption~\textbf{W*}.
It is obvious the nonnegative cost function $c$ is $\K$-inf-compact and the transition probabilities are
weakly continuous. Thus, Assumption~\textbf{W*} holds.

To verify the validity of Assumption~\textbf{B}. we calculate the relative value function $\ua.$

Let us calculate the value functions $\va$ for $\a\in [0,1).$
Since there is only one action at $n = 0, 1,\ldots,$ the infinite-horizon value function
\begin{align}
	\va (n) = \sum_{i = 0}^{\infty} \zone_{n + i} \a^i, \qquad n = 0,1,\ldots\ .
	\label{eqn:ex1:van>=0}
\end{align}
Therefore, for $n = 0,$ \eqref{eqn:ex1:van>=0} implies
\begin{align}
\begin{split}
	\va (0) &= \sum_{i = 0}^{\infty} \zone_{i} \a^i = z_0 + \sum_{i = 1}^{\infty} (z_i - z_{i-1})\a^i
		+ \sum_{i = 0}^{\infty} \a^i  \\
	&= \sum_{i = 0}^{\infty} z_i \a^i - \a \sum_{i = 0}^{\infty} z_i \a^{i} + \frac{1}{1-\a} = f(\a) + \frac{1}{1-\a} ,
\end{split}
	\label{eqn:ex1:van=0}
\end{align}
where the second equality follows from \eqref{eqn:ex1:defzn1}, the third equality is straightforward, and the last
equality follows from \eqref{eqn:ex1:deffalpha}. Furthermore, \eqref{eqn:ex1:van>=0} implies that for $n=1,2,\ldots$
\begin{align}
\begin{split}
	\va (n) &= \sum_{i=0}^{\infty} (z_{n+i} - z_{n+i-1} + 1) \a^i  = (1-\a) \sum_{i=0}^{\infty} z_{n+i} \a^i - z_{n-1} + \frac{1}{1-\a} .
\end{split}
	\label{eqn:ex1:van>0}
\end{align}
For the state $-1\in\X,$ let $\varphi^{(k)},$ $k=0,1,\ldots,$ be
the policy that takes action $a^s$ at steps $0,1,\ldots,k-1$ and $a^c$ at step $k$ and $\varphi^{(\infty)}$ be the
policy that always takes action $a^s.$ Then for $k = 0,1,\ldots$
\begin{align*}
	\va^{\varphi^{(\infty)}} (-1) = \frac{1}{1-\a} < \a^{k+1} [ f(\a) + \frac{1}{1-\a} ] + \frac{1-\a^{k+1}}{1-\a}
		= \va^{\varphi^{(k)}} (-1) ,
\end{align*}
where the inequality follows from $f(\a) > 0.$ Therefore,
\begin{align}
	\va (-1) = \frac{1}{1-\a}.
	\label{eqn:ex1:van=-1}
\end{align}
For state $-2 \in \X,$ \eqref{eqn:ex1:van=0} and \eqref{eqn:ex1:van=-1} imply
\begin{align}
	\va (-2) = \a \va (-1) = \frac{\a}{1-\a} \leq \va (-1) \leq \va (0) .
	\label{eqn:ex1:van=-2}
\end{align}
In view of \eqref{eqn:ex1:van>0} and \eqref{eqn:ex1:van=-2}, for $n=1,2,\ldots$
\begin{align}
	\va (n) \geq -1 + \frac{1}{1-\a} = \va (-2),
	\label{eqn:ex1:van>van=-2}
\end{align}
where the inequality follows from $z_n \in \{ 0,1 \},$ $n = 0, 1, \ldots\ .$ \eqref{eqn:ex1:van=-2}
and \eqref{eqn:ex1:van>van=-2} imply
\begin{align}
	\ma = \va (-2) = \frac{\a}{1-\a},
	\label{eqn:ex1:ma}
\end{align}
where $\ma := \inf_{x\in\X} \va (x)$ is defined in \eqref{defmauaw}. Thus, \eqref{eqn:ex1:van=0}--\eqref{eqn:ex1:van=-2}
and \eqref{eqn:ex1:ma} imply \eqref{eqn:ex1:ua}.

Note that $w^* \leq w^{\varphi^{(\infty)}} (-1) = 1 < \infty.$ Then to complete the proof of the validity of
Assumption~\textbf{B}, we need to prove that $\sup_{\a\in[0,1)} \ua (n) < \infty$ for $n\in\X.$
According to Feinberg et al.~\cite[Lemma 5]{FKZ12}, since the cost function $c \geq 0$ and $w^* < \infty,$
it is equivalent to prove that for $n\in \X$
\begin{align}
	\limsup_{\a\uparrow1} \ua (n) < \infty.
	\label{eqn:ex1:limitua<infty}
\end{align}
Since $0\leq z_n \leq 1,$ $n = 0,1,\ldots,$ \eqref{eqn:ex1:ua} implies that
$0 \leq \ua (n) \leq 1 + (1-\a) \sum_{i=0}^{\infty} \a^i = 2$ for $n\in\X$ and $\a\in[0,1).$
Hence, \eqref{eqn:ex1:limitua<infty} holds. This completes the proof.
\end{proof}

\begin{proof}[Proof of Lemma~\ref{lm:bua equicontinuous}]
\eqref{lm:bua-equi} For all $\a\in [0,1)$
\begin{align}
\begin{split}
	& |\bua (x) - \bua (y)| = |\bva (x) - \bva (y)| = |\va  (x) - \va  (y) + \c (x - y)| \\
	= & |\ua (x) - \ua (y) + \c (x - y) | \leq |\ua (x) - \ua (y)| + \c |x - y| ,
\end{split}
	\label{eqn:equicont-1}
\end{align}
where the first equality follows from \eqref{eqn:def bua},
the second one follows from Lemma~\ref{lm:samevalues}, and the third one follows from \eqref{defmauaw}.

Consider $\eps > 0.$ For each $\b\in (\a^*,1),$
since the family of functions  $\{ \ua \}_{\a\in [\b,1)}$ is equicontinuous
(see Theorem~\ref{thm:u a equicont}), there exists $\delta > 0$ such that
$|\ua (x) - \ua (y)| < \frac{\eps}{2}$ for all $|x-y|<\delta$ and $\a\in [\b,1).$ Therefore, for
$|x-y| < \delta_1 := \min\{ \delta, \frac{\eps}{2\c} \},$ $\c | x-y|<\frac{\eps}{2}$ and
\eqref{eqn:equicont-1} implies that for $|x-y|<\delta_1$ and $\a\in [\b,1)$
\begin{align*}
	|\bua (x) - \bua (y)| \leq \eps.
\end{align*}
Thus, the family of functions  $\{\bua\}_{\a\in [\b, 1)}$ is equicontinuous.

\eqref{lm:bua-bd} Consider $x\in\X.$ For all $\a\in (\a^*,1),$
\begin{align}
\begin{split}
	& \bua (x) - \ua(x) \leq |\bua (x) - \ua (x)|  \\
	= &  |\bva (x) - \va  (x) - (\bma - m_\a) | = |\c x -(\bma - m_\a)  |  \\
	\leq & \c |x| + |\bma - m_\a| \leq \c (|x| + |\sa| + |x_U^*|) \leq \c (|x| + b + |x_U^*|),
\end{split}
	\label{eqn:difbuaua}
\end{align}
where the last two inequalities follow from \eqref{eqn:bounds on bma} and Theorem~\ref{thm:limit s alpha} respectively.

According to Theorem~\ref{thm:B}, since Assumption ${\bf B}$ holds,
$\sup_{\a\in (\a^*,1)} \ua (x) < +\infty.$ Therefore, \eqref{eqn:difbuaua} implies that
\begin{align*}
\sup_{\a\in (\a^*,1)} \bua (x) \leq \c (|x| + b + |x_U^*|) + \sup_{\a\in (\a^*,1)} \ua (x) < +\infty,
\qquad x\in\X.
\end{align*}
\end{proof}

\begin{proof}[Proof of Lemma~\ref{lm:bua pointwise converge}]
Consider any given $\b\in (\a^*,1).$ According to \eqref{eqn:sab}, there exists $b>0$ such that
$\sa\in [-b,b]$ for all $\a\in [\b,1).$

Consider $s^*$ defined in \eqref{eqn:def s*}.
We first prove that the limit $\lim_{\a\uparrow1} \bu_{\a}(x)$ exists for $x<s^*.$
For $x < s^*,$ according to Theorem~\ref{thm:limit s alpha},
there exists $\hat{\a} > \b$ such that $\sa > x$ for all $\a \in [\hat{\a},1),$
then \eqref{eqn:bua-2} implies that $\bua (x) = K$ for all $\a \in [\hat{\a},1).$  Therefore,
\begin{align}
	\lim_{\a \uparrow1} \bua(x) = K, \qquad x < s^*.
	\label{eqn:bu x<s*}
\end{align}

For $x > s^*,$ according to Theorem~\ref{thm:limit s alpha}, there exists $\hat{\a} >\b$ such that
$\sa < x$ for all $\a \in [\hat{\a},1).$ Therefore, in view of
\eqref{eqn:ga:trans}, \eqref{eqn:def bua}, \eqref{eqn:bma}, and \eqref{eqn:bua-2}, for all $\a \in [\hat{\a},1)$
\begin{align}
	\bua (x) = \E[h_\a (x-D)] + \a \E[\bua (x-D)] - (1-\a) \bma .
	\label{eqn:bua x>s*-1}
\end{align}
Then, \eqref{eqn:bua-2} and \eqref{eqn:bua x>s*-1} imply that
\begin{align}
	\bua (x) = \E[\sum_{j = 1}^{\textbf{N}(x - \sa) + 1} \a^{j-1} (\tilde{h}_{\a} (x - \textbf{S}_{j-1}) - (1-\a)\bma)  ]
	+ \E[\a^{\textbf{N}(x - \sa) + 1} K ],
	\label{eqn:bua x>s*-2}
\end{align}
where $\tilde{h}_{\a} (x) := \E[h_{\a} (x-D)],$ $x\in\X,$ and $\textbf{N}(\cdot)$ is defined
in \eqref{renew}.
For $\a\in[\b,1),$
\begin{align}
	1\geq \E[\a^{\textbf{N}(x - \sa) + 1}] \geq \E[\a^{\textbf{N}(x + b) + 1}]\geq \a^{\E[\textbf{N}(x + b) + 1]},
	\label{eqn:boundaK1}
\end{align}
where the first inequality follows from $\a < 1,$ the second one follows from $\sa \geq -b$ and $\a < 1,$
and the last one follows from Jensen's inequality.
Since $\PR(D>0)>0,$ $\E[\textbf{N}(x + b) + 1]< +\infty,$ which implies that
\begin{align}
	\lim_{\a\uparrow1}\a^{\E[\textbf{N}(x + b) + 1]}\to 1.
	\label{eqn:boundaK2}
\end{align}
Therefore, \eqref{eqn:boundaK1} and \eqref{eqn:boundaK2} imply that
\begin{align}
	\lim_{\a\uparrow1} \E[\a^{\textbf{N}(x - \sa) + 1} K ] = K.
	\label{eqn:boundaK3}
\end{align}

The first term of the right hand side
of \eqref{eqn:bua x>s*-2} can be written as
\begin{align}
\begin{split}
	& \E[\sum_{j = 1}^{\textbf{N}(x - \sa) + 1} \a^{j-1} (\tilde{h}_{\a} (x - \textbf{S}_{j-1})- (1-\a)\bma)]  \\
	= & \sum_{i = 0}^{+\infty} \sum_{j=1}^{i+1} \a^{j-1}\E[\tilde{h}_{\a}(x- \textbf{S}_{j-1})- (1-\a)\bma|\textbf{N}(x - \sa) =i] \PR(\textbf{N}(x-\sa)=i)  \\
	= & \sum_{j = 0}^{+\infty} \sum_{i=j}^{+\infty} \a^j \E[\tilde{h}_{\a}(x- \textbf{S}_j)- (1-\a)\bma|\textbf{N}(x - \sa) =i] \PR(\textbf{N}(x-\sa)=i)  \\
	= & \sum_{j = 0}^{+\infty} \a^j\E[\tilde{h}_{\a}(x- \textbf{S}_j)- (1-\a)\bma|\textbf{N}(x - \sa)\geq j] \PR(\textbf{N}(x-\sa)\geq j)  \\
	= & \sum_{j = 0}^{+\infty} \a^j \E[\tilde{h}_{\a}(x-\textbf{S}_j)- (1-\a)\bma|\textbf{S}_j \leq x - \sa] \PR(\textbf{S}_j \leq x - \sa),
\end{split}
	\label{eqn:sum of h}
\end{align}
where the first and third equalities follow from the properties of conditional expectations, the second equality changes the order of summation, and its validity follows from the nonnegativity of $h(\cdot)$ and the finiteness of $\bua (x),$ and the last equality follows from the fact that
$\PR(\textbf{N}(t) \geq n) = \PR(\textbf{S}_n \leq t).$

Now we prove that, if $\a\uparrow1,$ then the limit of the expression in \eqref{eqn:sum of h} as $\a\uparrow1$ exists almost everywhere on $(s^*, +\infty).$
Consider the sets $\mathcal{D}_n,$ $n = 0,1,\ldots,$ on which the distribution function of
$\textbf{S}_n,$ is not continuous. Since every distribution function is right-continuous,
\begin{align*}
	\mathcal{D}_n := \{ x\in\R : \lim_{y\uparrow x} \PR(\textbf{S}_n \leq y) \neq \PR(\textbf{S}_n \leq x) \}.
\end{align*}
Therefore, each set $\mathcal{D}_n,$ $n=0,1,\ldots,$ is at most countably infinite. Let
\begin{align}
	\mathcal{D} = \{ x > s^*: x = s^* + y, \quad y\in \cup_{n=0}^{+\infty} \mathcal{D}_n   \}
	\qquad \text{and} \qquad
	\mathcal{C} = (s^*,+\infty) \setminus \mathcal{D}.
	\label{eqn:def of discont cont pts}
\end{align}
Hence, $\mathcal{D}$ is also at most countably infinite. In addition, $\PR(\textbf{S}_n \leq x -s^*)$ is continuous at $x-s^*$
and $\lim_{\a\uparrow1} \PR(\textbf{S}_n \leq x - \sa) = \PR(\textbf{S}_n \leq x -s^*)$ for all $x\in\mathcal{C}$ and $n=0,1,\ldots\ .$

Consider $x\in\mathcal{C}.$
Since the function $\h(x)$  is quasiconvex, $\a^j \E[\tilde{h}_{\a}(x-\textbf{S}_j)|\textbf{S}_j \leq x - \sa] \leq \E[\h(x - D)] + \E[\h (-b - D)] < +\infty.$
Since $\PR(D > 0)>0,$ there exists a real number $\Delta_D>0$ such that
$\PR(D > \Delta_D)>0.$ Consider
\begin{align*}
	\tilde{D} =
	\begin{cases}
		0 & \text{if } D < \Delta_D, \\
		\Delta_D & \text{otherwise.}
	\end{cases}
\end{align*}
Then $\E[\tilde{D}] = \Delta_D \PR(D\geq \Delta_D) >0$ and $Var(\tilde{D}) = \Delta_D^2 \PR(D\geq \Delta_D) (1-\PR(D \geq \Delta_D)) <  +\infty.$
Define $\tilde{\textbf{S}}_0 = 0$ and $\tilde{\textbf{S}}_n = \sum_{i=1}^n \tilde{D},$ $n=1,2,\ldots\ .$ Therefore, $\PR(\textbf{S}_n \leq x) \leq \PR(\tilde{\textbf{S}}_n \leq x)$ for all $x\in\R$ and $n=0,1,\ldots\ .$
Since $\E[\tilde{D}] > 0,$  there exists $N_1 > 0$ such that
$n\E[\tilde{D}] > x + b$ for all $n > N_1.$ Let $\Delta(n) := n\E[\tilde{D}] - (x+b) >0.$
Hence, for $n > N_1$
\begin{align*}
	& \PR(\textbf{S}_n \leq x - \sa) \leq \PR(\tilde{\textbf{S}}_n \leq x - \sa)
	\leq \PR(\tilde{\textbf{S}}_n \leq x + b) \\
	= & \PR(\tilde{\textbf{S}}_n-n\E[\tilde{D}]\leq x+b - n\E[\tilde{D}]) \leq  \PR(|\tilde{\textbf{S}}_n - n\E[\tilde{D}]| \geq \Delta(n) ) \leq \frac{Var(\tilde{D})}{\Delta^2 (n)},
\end{align*}
where the last inequality follows from Chebyshev's inequality. In addition, according to
Lemma~\ref{lm:limit m*}, there exists $M_1 > 0$ such that
$|(1-\a)\bma| \leq M_1$ for all $\a\in [\b,1).$  Therefore, the summation
$\sum_{j = N_1}^{+\infty} ( \E[\h(x - D)] + \E[\h (-b - D)] + M_1) \frac{Var(\tilde{D})}{\Delta^2 (j)} < +\infty.$ By taking the limit of the first and last terms of \eqref{eqn:sum of h} as $\a\uparrow1,$
\begin{align}
\begin{split}
	&\lim_{\a\uparrow1} \E[\sum_{j = 1}^{\textbf{N}(x - \sa) + 1} \a^{j-1} (\tilde{h}_{\a} (x - \textbf{S}_{j-1})- (1-\a)\bma) ]
 \\
	=& \sum_{j = 0}^{+\infty} \E[\tilde{h}_{1}(x-\textbf{S}_j)- w|\textbf{S}_j \leq x - s^*] \PR(\textbf{S}_j \leq x - s^*) \\
	=&  \E[\sum_{j = 1}^{\textbf{N}(x - s^*) + 1} (\tilde{h}_1 (x - \textbf{S}_{j-1})- w)],
\end{split}
	\label{eqn:limit of sum-2}
\end{align}
where the first equality follows from the Lebesgue's dominated convergence theorem, Theorem~\ref{thm:limit s alpha}, and Lemma~\ref{lm:limit m*},
and the second equality follows from \eqref{eqn:sum of h} with $\sa$ replaced with $s^*$ and $\a^j$ replaced with $1.$ In view of \eqref{eqn:bua x>s*-2}, \eqref{eqn:boundaK3}, and  \eqref{eqn:limit of sum-2},
\begin{align}
	\lim_{\a\uparrow 1} \bua (x) =
	\E[\sum_{j = 1}^{\textbf{N}(x - s^*) + 1} (\tilde{h}_1 (x - \textbf{S}_{j-1})- w)] + K, \qquad x\in\mathcal{C} .
	\label{eqn:convegC}
\end{align}

Let $\bar{\mathcal{D}}:= \mathcal{D}\cup \{s^*\}.$ The complement of $\bar{\mathcal{D}}$ is
$\bar{\mathcal{D}}^c = \X\setminus\bar{\mathcal{D}}  = (-\infty,s^*)\cup\mathcal{C}.$
In view of \eqref{eqn:bu x<s*} and \eqref{eqn:convegC}, there exists the limit
\begin{align}
	\bu (x) := \lim_{\a\uparrow1} \bua (x), \qquad x\in \bar{\mathcal{D}}^c.
	\label{eqn:convergDc}	
\end{align}

Now it remains to prove that \eqref{eqn:convergence-1} holds for $x\in \bar{\mathcal{D}}.$
In view of Lemma~\ref{lm:bua equicontinuous} and Arzel\`{a}-Ascoli theorem (see  Hern\'{a}ndez-Lerma and Lasserre~\cite[p.~96]{HLL96}), there exist a sequence
$\{ \a_n\uparrow1 \}_{n=1,2,\ldots}$ with $\a_1 \geq \b$ and a continuous function $\bar{u}^*$ such that
\begin{align}
	\lim_{n\to +\infty} \bu_{\a_n}(x) = \bu^* (x) , \qquad x\in\X,
	\label{eqn:bu*-1}
\end{align}
In view of \eqref{eqn:convergDc} and \eqref{eqn:bu*-1},
\begin{align}
	\bu (x) = \bu^* (x), \qquad x\in \bar{\mathcal{D}}^c.
	\label{eqn:bu=bu*}
\end{align}

The following proof is by contradiction. Consider $z\in\bar{\mathcal{D}}.$
Assume that there exists a sequence $\{ \gamma_n\uparrow1 \}_{n=1,2,\ldots}$ with $\gamma_1 \geq \b$ such that $\lim_{n\to  +\infty} \bu_{\gamma_n} (z) \neq \bu^* (z).$
According to Arzel\`{a}-Ascoli theorem, there exist a subsequence $\{ \gamma_{n_k} \}_{k=1,2,\ldots}$ of the sequence $\{ \gamma_n\uparrow1 \}_{n=1,2,\ldots}$ and a continuous function
$u^\prime$ such that
\begin{align}
	\lim_{k\to +\infty} \bu_{\gamma_{n_k}}(x) = \bu^\prime (x) , \qquad x\in\X .
	\label{eqn:bu*-2}
\end{align}
Therefore, $\bu^\prime (z) \neq \bu^* (z).$ However, in view of \eqref{eqn:convergDc}, \eqref{eqn:bu=bu*}, and \eqref{eqn:bu*-2}, $\bu^* (x) = \bu^\prime (x)$ for all $x\in \bar{\mathcal{D}}^c,$  which implies that
\begin{align*}
\bu^* (z) = \lim_{y\to z, y\in \bar{\mathcal{D}}^c} \bu^* (y) = \lim_{y\to z, y\in \bar{\mathcal{D}}^c} \bu^\prime (y) = \bu^\prime (z).
\end{align*}
Therefore, there exists the limit
\begin{align}
	\bu (x) := \lim_{\a\uparrow1} \bua (x), \qquad x\in \bar{\mathcal{D}}.
	\label{eqn:convergDc-1}	
\end{align}
Furthermore, \eqref{eqn:convergDc}, \eqref{eqn:bu*-1}, and \eqref{eqn:convergDc-1} implies
that $\bu = \bu^*$ and the function $\bu$ is continuous.
\end{proof}

\begin{proof}[Proof of Corollary~\ref{cor:inventory:acoe}]
	This corollary follows from Theorems~\ref{thm:inventory:acoe}, \ref{thm:limit s alpha} and \ref{thm:ua converge}.
\end{proof}

\begin{proof}[Proof of Corollary~\ref{cor:inventory convex}]
	This corollary holds because Assumption~\ref{assum:convex} implies Assumptions~\ref{assum:qcall} and
	\ref{assum:decrease}; see Lemma~\ref{lm:cqc}
\end{proof}

\end{appendices}


\begin{thebibliography}{}

\bibitem{BFZ14}
	Bishop~C.~J., Feinberg~E.~A., \& Zhang~J. (2014).
	Examples concerning Abel and Ces\`{a}ro limits.
	{\em Journal of Mathematical Analysis and Applications,} 420(2), 1654--1661.

\bibitem{BCST}
	Beyer~D., Cheng~F., Sethi~S.~P., \& Taksar~M. (2010).
	{\em Markovian demand inventory models,}
	Springer, New York.

\bibitem{BS99}	
	Beyer~D. \& Sethi~S. (1999).
	The classical average-cost inventory models of Iglehart and Veinott--–Wagner revisited.
	{\em Journal of Optimization Theory and Applications,} 101(3), 523--555.

\bibitem{CS04a}
	Chen~X. \& Simchi-Levi~D. (2004).
	Coordinating inventory control and pricing strategies with random
	demand and fixed ordering cost: the finite horizon case.
	{\em Oper. Res.,} 52(6), 887--896.

\bibitem{CS04b}
	Chen~X. \& Simchi-Levi~D. (2004).
	Coordinating inventory control and pricing strategies with random
	demand and fixed ordering cost: the infinite horizon case.
	{\em Math. Oper. Res.,} 29(3), 698--723.


\bibitem{Ftut}
	Feinberg~E.~A. (2016).
	Optimality conditions for inventory control. In A. Gupta \& A. Capponi (Eds.),
	{\em Tutorials in operations research. Optimization challenges in complex, networked, and risky systems}
	(pp. 14--44), INFORMS, Cantonsville, MD.

\bibitem{FKZ12}
    Feinberg~E.~A., Kasyanov~P.~O., \& Zadoianchuk~N.~V. (2012).
    Average cost Markov decision processes with weakly continuous
    transition probability.
    {\em Math. Oper. Res.,} 37(4), 591--607.

\bibitem{FKZ13}
    Feinberg~E.~A., Kasyanov~P.~O., \& Zadoianchuk~N.~V. (2013).
    Berge's theorem for noncompact image sets.
    {\em J. Math. Anal. Appl.,} 397(1), 255--259.

\bibitem{FKZ16}
	Feinberg~E.~A., Kasyanov~P.~O., \& Zgurovsky~M.~Z. (2016).
	Partially observable total-cost Markov decision processes with weakly continuous transition probabilities. 	
	{\em Math. Oper. Res.,} 41(2), 656--681.

\bibitem{FL07}
    Feinberg~E.~A. \& Lewis~M.~E. (2007)
    Optimality inequalities for average cost Markov decision processes and the stochastic cash balance problem.
    {\em Math. Oper. Res.,} 32(4), 769--783.

\bibitem{FL16}
    Feinberg~E.~A. \& Lewis~M.~E. (2017).
    On the convergence of optimal actions for Markov decision processes and the optimality
    of $(s,S)$ inventory policies.
    {\em Naval Research Logistic,} DOI:10.1002/nav.21750.

\bibitem{FLi16a}
	Feinberg~E.~A. \& Liang~Y. (2017).
	Structure of optimal policies to periodic-review inventory models with convex costs and backorders for all values of discount factors. {\em Annals of Operations Research,} DOI:10.1007/s10479-017-2548-6.
	
\bibitem{FLi16b}
	Feinberg~E.~A. \& Liang~Y. (2017).
	On the optimality equation for average cost Markov decision processes and its validity for inventory control.
	{\em Annals of Operations Research,} DOI:10.1007/s10479-017-2561-9.

\bibitem{GKLSS}
	Goldberg~D.~A., Katz-Rogozhnikov~D.~A., Lu~Y., Sharma~M., \& Squillante~M.~S. (2016).
	Asymptotic optimality of constant-order policies for lost sales inventory models with large lead times.
	 {\em Mathematics of Operations Research,} 41(3), 898--913.

\bibitem{Her89}
	Hern\'{a}ndez-Lerma~O. (1989).
	{\em Adaptive Markov Control Processes.}
	Springer-Verlag, New York.

\bibitem{HLL96}
    Hern\'{a}ndez-Lerma~O. \& Lasserre~J.~B. (1996).
    {\em Discrete-time Markov control processes: basic optimality creteria.}
    Springer-Verlag, New York.

\bibitem{HJN11}
	Huh~W.~T., Janakiraman~G., \& Nagarajan~M. (2011)
	Average cost single-stage inventory models: an analysis using a vanishing discount approach.
	{\em Operations Research,} 59(1), 143--155.

\bibitem{Igl63}
	Iglehart~D.~L. (1963)
	 Dynamic programming and stationary analysis of inventory problems. In
	\emph{Office of Naval Research Monographs on Mathematical Methods in Logistics} (H. Scarf, D. Gilford, and M. Shelly, eds.),
	 pp. 1--31, Stanford University Press, Stanford, CA.

\bibitem{Por02}
	Porteus~E. (2002).
	{\em Foundations of stochastic inventory theory.}
	Stanford University Press, Stanford, CA.

\bibitem{Res92}
	Resnick~S.~I. (1992).
	{\em Adventures in stochastic processes.}
	Birkhauser, Boston.

\bibitem{Sca60}
	Scarf~H. (1960).
	The optimality of (S, s) policies in the dynamic inventory problem.
	{\em Mathematical Methods in the Social Sciences} (K. Arrow, S. Karlin, and P. Suppes, eds.),
	Stanford University Press, Stanford, CA.

\bibitem{Sch93}
	Sch\"{a}l~M. (1993).
	Average optimality in dynamic programming with general state space.
	{\em Math. Oper. Res.,} 18(1), 163--172.
	
\bibitem{SCB05}
	Simchi-Levi~D., Chen~X., \& Bramel~J. (2005).
	{\em The Logic of Logistics: Theory, Algorithms, and Applications for Logistics and Supply Chain Management.}
	Springer-Verlag, New York.


\bibitem{Vei66}
	Veinott~A.~F. (1966).
	On the optimality of $(s, S)$ inventory policies: new condition and a new proof.
	{\em J. SIAM Appl. Math.}, 14(5), 1067--1083.
	
\bibitem{VW65}
	Veinott~A.~F. \& Wagner~H.~M. (1965).
	Computing optimal $(s,S)$ policies.
	{\em Management Science,} 11(5), 525--552.

\bibitem{za62}
	Zabel~E. (1962).
	A note on the optimality of $(s,S)$ policies in inventory theory.
	{\em Manag. Sci.,} 9(1), 123--125.

\bibitem{Zheng}
	Zheng~Y. (1991).
	A simple proof for optimality of $(s,S)$ policies in infinite-horizon inventory systems.
	{\em J. Appl. Prob.}, 28(4), 802--810.

\bibitem{Zip00}
	Zipkin~P.~H. (2000).
	{\em Foundations of inventory management.}
	McGraw-Hill, New York.


\end{thebibliography}
\end{document}